\documentclass[reqno]{amsart}
\usepackage[xcolor]{changebar}
\usepackage{latexsym}
\usepackage{amssymb,amsthm,amsmath}
\usepackage[dvips]{graphicx}
\usepackage{changebar}
\cbcolor{lightgray}

\allowdisplaybreaks 

\theoremstyle{plain}
\newtheorem{theorem}{Theorem}
\newtheorem{lemma}[theorem]{Lemma}
\newtheorem{corollary}[theorem]{Corollary}

\theoremstyle{definition}
\newtheorem{definition}[theorem]{Definition}
\theoremstyle{remark}
\newtheorem{remark}[theorem]{Remark}

\DeclareMathOperator{\Div }{div}

\def\pa{\partial}
\def\cal{\mathcal}
\let\mib=\boldsymbol

\def\R{{\mathbb R}}

\def\eps{\varepsilon}

\def\mx{{\bf x}}
\def\mz{{\bf z}}
\def\mxi{{\mib \xi}}
\def\my{{\bf y}}
\def\mz{{\bf z}}

\def\mff{{\mathfrak f}}

\newcommand{\tensorkl}{\mathcal{T}^k_l}
\newcommand{\tensoroo}{\mathcal{T}^1_1}
\newcommand{\supp}{\mathrm{supp}\,}
\newcommand{\mcf}{{\mathcal F}}

\begin{document}

\title[Well-posedness theory on Riemannian manifolds]{Well-posedness theory for degenerate parabolic equations on Riemannian manifolds}
\author{M.\ Graf}\address{Melanie Graf, University of Vienna, 
Faculty of Mathematics}\email{melanie.graf@univie.ac.at}
\author{M.\ Kunzinger}\address{Michael Kunzinger, University of Vienna, 
Faculty of Mathematics}\email{michael.kunzinger@univie.ac.at}
\author{D.\ Mitrovic}\address{Darko Mitrovic,
Faculty of Mathematics, University of
Montenegro }\email{darkom@ac.me}



\begin{abstract}
We consider the degenerate parabolic equation
$$
\partial_t u +\mathrm{div} {\mathfrak f}_{\bf x}(u)=\mathrm{div}(\mathrm{div} ( A_{\bf x}(u) ) ), 
\ \ {\bf x} \in M, \ \ t\geq 0
$$ 
on a smooth, compact, $d$-dimensional Riemannian manifold $(M,g)$. Here, for each $u\in \R$, ${\bf x}\mapsto {\mathfrak f}_{\bf x}(u)$  is a vector field and ${\bf x} \mapsto A_{\bf x}(u)$ is a $(1,1)$-tensor field on $M$ such that $u\mapsto \langle A_{\bf x}(u) {\mib \xi},{\mib \xi} \rangle$, 
${\mib \xi} \in T_\mx M$,  is non-decreasing with respect to $u$. The fact that the notion of divergence appearing in the equation depends on the metric $g$ requires revisiting the standard entropy admissibility concept. We derive it under an additional geometry compatibility condition and, as a corollary, we introduce the kinetic formulation of the equation on the manifold. Using this concept, we prove well-posedness of the corresponding Cauchy problem. 
\end{abstract}
\maketitle

\section{Introduction}

We consider the Cauchy problem for a degenerate parabolic equation of the form
\begin{align}
\label{main-eq}
\pa_t u +\Div \mff_{\mx}(u)&=\Div(\Div ( A_{\mx}(u) ) ), \ \ \mx \in M, \ \ t\geq 0\\
\label{ic}
u|_{t=0} &=u_0(\mx) \in L^\infty(M)
\end{align} 
on a smooth (Hausdorff), compact, $d$-dimensional Riemannian manifold $(M,g)$. For simplicity, we shall assume that 
\begin{equation}
\label{id-as-1}
0\leq u_0 \leq 1.
\end{equation} 
This is a natural assumption since equation \eqref{main-eq}, among other phenomena, describes fluid concentration dynamics in the case of flow in porous media (Buckley-Leverett type equations), and concentration always varies between zero and one (see e.g. \cite{BV}).

We suppose that the map $(\mx,\xi)\mapsto \mff_{\mx}(\xi) \equiv \mff(\mx,\xi)$, 
$M\times \R \to TM$ is $C^1$ and that, for every $ \xi\in \R$,
$\mx \mapsto \mff_{\mx}(\xi)\in \mathfrak{X}(M)$ (the space of vector fields on $M$).
Also, $(\mx,\xi)\mapsto A_\mx(\xi): M\times\R \to T^1_1M$ is supposed to satisfy   $\mx\mapsto A_\mx(\xi)\in \mathcal{T}^1_1(M)$ for each $\xi\in \R$ 
and we assume that the  $\xi$-derivative of $A$ is positive semi-definite and
\begin{equation}\label{Apd}
A_\mx'(\xi)=\sigma_\mx(\xi)^\top \sigma_\mx(\xi),
\end{equation} 
with
$\sigma$ such that $(\mx,\xi)\mapsto \sigma_\mx(\xi): M\times\R \to T^1_1M$ is $C^2$ and 
$\mx\mapsto \sigma_\mx(\xi)\in \mathcal{T}^1_1(M)$ for each $\xi\in \R$.  Here $\sigma^\top\in \tensoroo(M)$ denotes the transpose of $\sigma\in \tensoroo(M)$, i.e.,
 the unique tensor field such that $\langle \sigma(X),Y\rangle = \langle X,\sigma^\top(Y)\rangle$
for any $X,\ Y\in \mathfrak{X}(M)$.
In particular, this implies that $\xi \mapsto \langle A_{\mx}(\xi) \mxi,\mxi \rangle$ 
is non-decreasing for any $\mxi\in T_\mx M$. 

In local coordinates, we write
\begin{align*}
\mff_{\mx}(\xi)=(f^1(\mx,\xi),\dots,f^d(\mx,\xi)), \ \ A_{\mx}(\xi)=(A_{j}^k(\mx,\xi))_{k,j=1,\dots,d}.
\end{align*}  
The divergence operator appearing in the equation is to be formed with respect to the metric,
so in local coordinates we have (cf.\ \eqref{divx} below):
\begin{equation}
\label{div-g}
\Div \mff_{\mx}(u)=\Div \big(\mx \mapsto \mff_{\mx}(u(t,\mx))\big)=\frac{\pa}{\pa x_k} (f^k_{\mx}(u(t,\mx))+\Gamma^j_{kj}(\mx) f^k_{\mx}(u(t,\mx))
\end{equation} 
where the $\Gamma$-terms are the Christoffel symbols of $g$ and the Einstein summation convention is in effect. 
Similarly, the right hand side of \eqref{main-eq} is to be understood as
\begin{equation}
\label{div-A} 
\Div(\mx\mapsto \Div ( A_{\mx}(u(t,\mx)))), 
\end{equation}  whose explicit local expression can be read off from \eqref{divdiv11} below.

Equation \eqref{main-eq} describes a flow governed by

\begin{itemize}

\item the convection effects (bulk motion of particles), which are
represented by the first order terms, i.e.\ by the flux $\mff$;

\item diffusion effects, which are represented by the second order
term, i.e., the $(1,1)$-tensor
$A_{\mx}(\xi)$  
(more precisely its derivative with respect to $\xi$, denoted by $a$; see
\eqref{stand}) which describes direction and intensity of the diffusion of, e.g., a fluid whose 
concentration at $\mx \in M$ at time $t\geq 0$ is $u(t,\mx)$.

\end{itemize} The equation is degenerate in the sense that $\pa_\xi A_\mx$ can be equal to zero in some direction for some $\mx\in M$ 
(i.e., $\langle A(\mx,\xi)\mxi,\mxi\rangle$ is not strictly increasing with respect to $\xi$).
Roughly speaking, if this is the case (i.e., if for some vector $\mxi
\in T_\mx M$ we have $\langle \pa_\xi A(\mx,\xi)\mxi,\mxi\rangle =0$), 
then diffusion effects do not exist at the point $\mx$ for the state
$\xi$ in the direction $\mxi$. 

We note that the usual form of a degenerate parabolic equation (see e.g. \cite{CK}) is
\begin{equation}
\label{stand}
\pa_t u +\Div \mff(\mx,u)=\Div(a(\mx,u) \nabla u ) ).
\end{equation} 
In the flat case (i.e., when $M=\R^d$ with the Euclidean metric), 
equation \eqref{main-eq} is obviously reduced to \eqref{stand} simply by putting $a(\xi)=A'(\xi)$,
where the prime denotes the derivative with respect to $\xi$ (with slightly more algebra, one can show that this also holds when $A$ depends on $(t,\mx)$ as well). However, form \eqref{stand} is not convenient for deriving the entropy conditions given in Definition \ref{def-entropy}.  

To resolve this
problem we follow the foundational works \cite{CK, CP} in introducing an appropriate
entropy admissibility concept for \eqref{main-eq}
under the following geometry compatibility  condition  (see \cite{BLf} for an appropriate notion in the case of scalar conservation laws):
\begin{equation}
\label{geomcomp}
\Div \mff_{\mx}(\xi)=\Div (\Div (A_{\mx}(\xi)) \ \ \text{for every $\xi\in \R$}.
\end{equation} We note that, from a physical point of view, this is an 
incompressibility condition (divergence of the (diffusive) flux $\mff_{\mx}(\xi)-\Div (A_{\mx}(\xi))$ is zero). Indeed, conservation of mass of an incompressible fluid in a control volume changes the density only due to the diffusion effects:
\begin{equation}
\label{1}
\frac{D\rho }{Dt}={\Div}(A'(x,\rho)\cdot\nabla \rho), \ \ A'(x,\rho)=\pa_\xi A(x,\xi)\big{|}_{\xi=\rho},
\end{equation}where $\rho$ is density of the control volume and $\frac{D\rho }{Dt}=\frac{\pa \rho}{\pa t}+\frac{d\mx}{dt} \cdot \nabla \rho$ is the material derivative for the flow velocity $\frac{d\mx}{dt}=(\frac{dx_1}{dt},\dots,\frac{dx_d}{dt})$. If we rewrite our equation in $\R^d$ (with the Euclidean metric, writing $\rho$ instead of $u$ and disregarding non-smoothness for the moment), we actually have
\begin{equation}
\begin{split}
\frac{\pa \rho}{\pa t}+\pa_\xi \big(f(\mx,\xi)-{\Div}A(\mx,\xi)\big)\big{|}_{\xi=\rho} \cdot \nabla \rho & +  {\rm Div} (f(\mx,\xi)-{\Div}A(\mx,\xi))\big{|}_{\xi=\rho}\\
&={\Div} (A'(\mx,\rho) \cdot \nabla \rho)
\label{2}
\end{split}
\end{equation} Then, taking as usual 
$
\frac{d\mx}{dt}=\pa_\xi \big(f(\mx,\xi)-{\Div}A(\mx,\xi)\big)\big{|}_{\xi=\rho}
$ and comparing \eqref{2} and \eqref{1}, we arrive at
$$
(\Div f(\mx,\xi)-\Div ( \Div ( A(\mx,\xi))))\big{|}_{\xi=\rho}=0,
$$ 
which immediately gives what we called the geometry compatibility condition.

Since the equation we consider is of degenerate parabolic type, solutions are not necessarily
smooth and weak solutions must be sought. Such a weaker solution concept may result in non-uniqueness, and so we need to eliminate ``non-physical'' solutions through an entropy admissibility concept (\cite{CK, CP}).

With appropriate admissibility conditions in place, we can fairly directly derive the kinetic formulation to \eqref{main-eq} (see \eqref{kin-eq-1}). This  generalization of similar previous results (\cite{CP, D, D1}) is, however, not enough to provide well-posedness of admissible solutions to \eqref{main-eq}. What has to be incorporated in the kinetic formulation is the chain rule (see Theorem \ref{chain}), originally introduced in \cite{CP}, and extended to the heterogeneous setting in \cite{CK}. We implement this in a general way, which does not presuppose the form of the kinetic function (see the comments after Remark \ref{rem5} below), and which may generate several stable semigroups of solutions (compare standard and non-standard shocks, for instance in \cite{vDP, Lfl}). We also note that our kinetic solution concept for degenerate parabolic equations is new also from the standard Euclidean point of view.

Degenerate parabolic equations appear in a broad spectrum of applications, such as sedimentation-consolidation processes (\cite{BCBT}) or flow in porous media (\cite{EK}), which very often occur in non-flat media (e.g., during the CO${}_2$ sequestration
process the caprock confining the brine in which gas is injected is basically never flat, cf.\ \cite{JMN}). In other words, in our situation, we consider a flow governed by the convection and diffusion effects along a non-flat surface. 

Nevertheless, due to obvious technical complexities, the equation was so far only considered on the entire space (see e.g.\ \cite{CK, CP, Db} and references therein). Moreover, while the existence problem was settled a fairly long time ago \cite{VH}, uniqueness in the case of an anisotropic diffusion was obtained only rather recently in \cite{CP} for homogeneous coefficients, and in \cite{CK} for the heterogeneous ones. Our strategy of proof follows the one developed in \cite{CP}. However, unlike the situation from these works, where the kinetic formulation is used only to prove uniqueness of solutions, here we develop the concept so that it can be used for the existence proof as well. This is in accordance with the standard kinetic approach used for conservation laws when the weak convergence of the kinetic functions (\cite{BV, MN, Pan_tams}) (or the Young measures (\cite{BLf, Dpe}), which is essentially equivalent) corresponding to a sequence of approximate solutions together with uniqueness of the kinetic function provide well-posedness of entropy solutions to \eqref{main-eq}, \eqref{ic}.

Although investigations concerning well-posedness of evolution equations on manifolds attracted a significant amount of attention recently, this problem for degenerate parabolic equations on manifolds has not been considered until now. The most closely related research is directed towards scalar conservation laws on manifolds and we mention \cite{BLf, LM, pan} for the Cauchy problem corresponding to scalar conservation laws on manifolds, and \cite{KMS, Pan_tams} for the (initial)-boundary value problem on manifolds. The approach in \cite{Pan_tams} is based on the kinetic formulation as well, and Definition 3.1. from there inspired our kinetic solution concept.

The paper is organized as follows. In Section \ref{diffprelsec} we introduce notions and notations from differential geometry as well as the entropy admissibility concept corresponding to \eqref{main-eq}. We then move on to derive the kinetic formulation of \eqref{main-eq}. In Section \ref{uniquesec}, we prove a uniqueness result for the kinetic formulation of the problem under consideration. Finally, in Section \ref{existencesec} we show existence of kinetic solutions as well as existence and uniqueness of entropy solutions.

\section{Preliminaries from Riemannian geometry and the entropy admissibility concept}\label{diffprelsec}

Our standard references for notions from Riemannian geometry are \cite{ON83,Pet06}.  For notions
and results from distributional geometry we refer to \cite{Mar,GKOS}. 
As already stated in the introduction,
$(M,g)$ will be a $d$-dimensional Riemannian manifold.  If $v$ is a distributional vector field on $M$  
then its gradient $\nabla v$ is the
vector field metrically equivalent to the exterior derivative $dv$ of $v$: $\langle \nabla v, X\rangle = dv(X) = X(v)$
for any $X\in \mathfrak{X}(M)$. 
In local coordinates,
\begin{equation}
\label{nabla}
\nabla v = g^{ij} \frac{\pa v}{\pa x^i} \pa_{j},
\end{equation} 
with $g^{ij}$ the inverse matrix to $g_{ij}=\langle \pa_{x^i},\pa_{x^j}\rangle$.
For $T\in \tensorkl(M)$, a divergence of $T$ is any contraction of one of its $k$ contravariant slots with
the new covariant slot of its covariant differential $\nabla T\in \mathcal{T}^k_{l+1}(M)$. 
In particular, if $k=1$ then $T$ possesses a unique divergence $\Div T \in \mathcal{T}^0_{l}(M)$. 
We list here the local coordinate expressions for the cases that will be of interest in
this paper. 

First, if $X\in \mathcal{T}^1_0=\mathfrak{X}(M)$ is a $C^1$ vector field on
$M$ with local representation $X=X^i\frac{\pa}{\pa x^i}$, then $\Div X\in C(M)$
is locally given by
\begin{equation}\label{divx}
\Div X = \frac{\pa X^k}{\pa x^k} + \Gamma^j_{kj}X^k.
\end{equation}
 The same expression holds for $X$ a distributional vector field, and similar for the formulae given below, which  
we formulate in the smooth case with the understanding that they carry over by continuous extension also to 
the distributional setting. 
If a $C^1$ one-form $\omega\in \mathcal{T}^0_1(M)=\Omega^1(M)$ is locally given by $\omega=\omega_idx^i$, then
its divergence is defined as the {\em metric} contraction of its covariant differential $\nabla \omega\in \mathcal{T}^0_2(M)$, so
\begin{equation} \label{divom}
\Div \omega = g^{ij} \pa_i\omega_j - \Gamma^k_{il} g^{il}\omega_k.
\end{equation} If $T\in \tensoroo(M)$, $T=T^k_i \frac{\pa}{\pa x^k}\otimes dx^i$, then $\Div T = (\Div T)_i dx^i$, where
\begin{equation}\label{div11}
(\Div T)_i = \pa_j T^j_i + \Gamma^j_{jl} T^l_i - \Gamma^l_{ji} T^j_l.
\end{equation}
Finally, again for $T\in \tensoroo(M)$, $\Div(\Div(T))\in C(M)$ is given in local coordinates by
\begin{align}\label{divdiv11}
\begin{split}
\Div(\Div(T)) = g^{ij}\big[  \pa_i\pa_kT^k_j + \Gamma^k_{kl} \pa_i T^l_j - &\Gamma^l_{kj}\pa_i T^k_l - \Gamma^k_{ij}\pa_l T^l_k
+(\pa_i \Gamma^k_{kl})T^l_j  \\
&  - (\pa_i \Gamma^l_{kj})T^k_l -
\Gamma^k_{ij}\Gamma^l_{lr}T^r_k + \Gamma^k_{ij}\Gamma^r_{kl}T^l_r \big]
\end{split}
\end{align}
In the Cauchy problem \eqref{main-eq}, \eqref{ic}, $(\mx,\xi)\mapsto A_\mx(\xi): M\times\R \to T^1_1M$ is $C^1$ and 
for each $\xi\in \R$, $\mx\mapsto A_\mx(\xi)\in \mathcal{T}^1_1(M)$. In general, if $T$ is a $(1,k)$-tensor with
$C^1$-dependence on an additional real variable $\xi$, i.e., $(\mx,\xi)\mapsto T_\mx(\xi): M\times\R \to T^1_kM$ is $C^1$ and 
for each $\xi\in \R$, $\mx\mapsto T(\mx,\xi)\in \mathcal{T}^1_k(M)$, then (recalling that the derivative with respect to $\xi$
is denoted by $T'$), it follows from the chain rule and the corresponding local expressions that
for an $H^1\cap L^\infty$-function $u: \R\times M \to \R$, we have
\begin{align}
\label{tensor1k}
\begin{split}
\Div(T(\mx,u(t,\mx)))_{i_1,\dots,i_k}-&\Div(\mx \mapsto T(\mx,\xi))_{i_1,\dots,i_k}\Big|_{\xi=u(t,\mx)}\\
&= T'^j_{i_1,\dots,i_k}(\mx,u(t,\mx)) \pa_j u(t,\mx).
\end{split} 
\end{align}
Furthermore, if $(\mx,\xi) \mapsto \omega(\mx,\xi): M \times \R \to T^0_1 M$ is such that for every $\xi \in \R$ it holds
that  $\mx \mapsto T(\mx,\xi) \in {\cal T}_1^0(M)=\Omega^1(M)$, we obtain from \eqref{divom}
\begin{equation}
\label{tensor01}
\Div(\omega(\mx,u(t,\mx))-\Div(\mx\mapsto \omega(\mx,\xi))|_{\xi=u(t,\mx)}=g^{ij}(\mx) \omega'_i(\mx,u(t,\mx))\pa_j u(t,\mx).
\end{equation} 
Finally, by $H^{1,2}(\R^+\times M)$ we denote the Sobolev space of order $1$ in the $t$-variable and of order $2$ in the
$\mx$-variable.
After these preparations we can prove:
\begin{theorem}
\label{thm1}Assume that the compatibility condition \eqref{geomcomp} holds and that $u: \R^+\times M \to \R$ is a bounded $H^{1,2}(\R^+\times M)$ non-negative solution to 
\eqref{main-eq}. Then for any $S\in C^2(\R)$ such that $S(0)=0$ we have
\begin{align}\label{a1}
\begin{split}
\pa_t S(u)+ \Div &\int_0^{u(t,\mx)} \mff'_{\mx}(\xi)S'(\xi)\, d\xi\\
   &=\Div \Div \Big(\int_0^{u(t,\mx)} A_{\mx}'(\xi)S'(\xi)\, d\xi\Big)-S''(u) \langle A'_{\mx}(u) \nabla u,\nabla u \rangle,
\end{split}
\end{align} 
where $\mff'=\pa_\xi \mff$ and $A'=\pa_\xi A$.
\end{theorem}
\begin{proof}
First, note that for any $f\in C^1(M;\R)$, $\omega \in {\cal T}_1^0(M)$ we have
\begin{equation}
\label{f1}
\Div (f \omega)= f \Div \omega+ g^{ij}\pa_j f \omega_i= f\Div \omega + \omega(\nabla f).
\end{equation} 
Based on this, we calculate for any $S\in C^2(\R)$ such that $S(0)=0$ (keeping in mind that $u$ is non-negative):
\begin{align}\label{f2}
\begin{split}
\Div \Big( &\int_0^{u(t,\mx)} \mff_{\mx}'(\xi) S'(\xi)\, d\xi \Big)\\
&=S'(u(t,\mx)) {f'}^i(\mx,u(t,\mx))\pa_i u(t,\mx)+\int_0^{u(t,\mx)}S'(\xi) \pa_i {f'}^i(\mx,\xi)\,d\xi \\
&\qquad\qquad\qquad\qquad\qquad\qquad\qquad\quad+\int_0^{u(t,\mx)}S'(\xi) \Gamma^j_{kj} {f'}^k(\mx,\xi)\,d\xi\\
&\overset{\eqref{tensor1k}}{=} S'(u(t,\mx))\Div (\mff_{\mx}(u(t,\mx))-S'(u(t,\mx))\Div (\mff_{\mx}(\xi))\Big|_{\xi=u(t,\mx)} \\
&\qquad\qquad\qquad\qquad\qquad\qquad\qquad\quad+\int_0^{u(t,\mx)}S'(\xi) \Div (\mff'(\mx,\xi))\,d\xi.
\end{split}
\end{align} 
Also, 
\begin{align}
\label{f3}
\begin{split}
\Big( \Div & \int_0^{u(t,\mx)}A'_{\mx}(\xi) S'(\xi)d\xi \Big)_{i}\\
&\overset{\eqref{tensor1k}}{=} {A'}^j_i(\mx,u(t,\mx))S'(u(t,\mx))\pa_j u +\Div \left(\int_0^\xi A'_{\mx}(v) S'(v)dv \right)_i\Big|_{\xi =u(t,\mx)} \\ 
& \overset{\eqref{tensor1k}}{=} S'(u(t,\mx))\Div(A_{\mx}(u(t,\mx)))_i-S'(u(t,\mx))\Div(A_\mx(\xi))_i\Big|_{\xi=u(t,\mx)}  \\
& \qquad\qquad\qquad\qquad\qquad\qquad\qquad\qquad+\int_0^{u(t,\mx)}\Div(A'_{\mx}(\xi))_i S'(\xi) d\xi.
\end{split}
\end{align} Now set $\tilde{\omega}(\mx,\xi):= \Div \int_0^\xi A'_{\mx}(v)S'(v)dv$ and $\bar{\omega}(\mx,\xi):=\Div(A_\mx(\xi))$.\ Using this notation and applying \eqref{f1} to the first two terms on the right-hand side of \eqref{f3}, we obtain
\begin{align}
\label{f4}
\begin{split}
\Div (\tilde{\omega}(\mx,&u(t,\mx))\!=\!S'(u(t,\mx))\Div \Div (A_{\mx}(u(t,\mx)))\\
&+\!g^{ij}\Div(A_{\mx}(u(t,\mx))_i S''(u(t,\mx))\pa_j u-S'(u(t,\mx))\Div(\bar{\omega}(\mx,u(t,\mx))
\\
&-g^{ij}\bar{\omega}_i(\mx,u(t,\mx))S''(u(t,\mx))\pa_j u+\Div \int_0^{u(t,\mx)} \Div(A'_{\mx}(\xi) S'(\xi) d\xi.
\end{split}
\end{align} 
Here,
\begin{align}
\label{f5}
\begin{split}
&\Div \int_0^{u(t,\mx)} \Div(A'_{\mx}(\xi)) S'(\xi)\, d\xi \\
&\overset{\eqref{tensor01}}{=} g^{ij} \Div (A'_{\mx}(\xi))\Big|_{\xi=u(t,\mx)} S'(u(t,\mx))\pa_ju
+\Div\!\! \int^\xi_0\Div(A'_{\mx}(v))S'(v)dv \Big|_{\xi=u(t,\mx)}  \\
&=S'(u(t,\mx))g^{ij}\bar{\omega}'_i(\mx,u(t,\mx))\pa_j u +\int_0^{u(t,\mx)}\Div \Div (A'_{\mx}(\xi))S'(\xi)\,d\xi \\
&\overset{\eqref{tensor01}}{=}S'(u(t,\mx))\left(\Div(\bar{\omega}(t,u(t,\mx)))-\Div(\bar{\omega}(\mx,\xi))\Big|_{\xi=u(t,\mx)} \right) \\
&\qquad\qquad\qquad\qquad\qquad\qquad\qquad+\int_0^{u(t,\mx)}\Div\Div(A'_{\mx}(\xi))S'(\xi)\,d\xi.
\end{split}
\end{align} 
From \eqref{f4} and \eqref{f5}, we conclude
\begin{align}
\label{f6}
\begin{split}
\Div &\Div \Big(\int_0^{u(t,\mx)}
A'_{\mx}(\xi) S'(\xi)\,d\xi \Big)\\
&=S'(u(t,\mx))\Div \Div (A_\mx(u(t,\mx)))-S'(u(t,\mx)) \Div\Div (A_\mx(\xi))\Big|_{\xi=u(t,\mx)}\\
&+S''(u(t,\mx))g^{ij}{A'}^r_{i}(\mx,u(t,\mx))\pa_r u \pa_j u+\int_0^{u(t,\mx)}\Div \Div(A'_{\mx}(\xi)) S'(\xi)\,d\xi \\
&=S'(u(t,\mx))\Div \Div(A_\mx(u(t,\mx)))-S'(u(t,\mx))\Div\Div (A_\mx(\xi)\Big|_{\xi=u(t,\mx)} \\
&+S''(u(t,\mx))\langle A'_{\mx}(u(t,\mx)) \nabla u,\nabla u \rangle+\int_0^{u(t,\mx)}\Div \Div (A'_{\mx}(\xi)) S'(\xi)\,d\xi. 
\end{split}
\end{align} 
Finally, 
\begin{align*}
&\Div (\int_0^{u(t,\mx)}\mff'_{\mx}(\xi)S'(\xi)\,d\xi)\overset{\eqref{f2}}{=} S'(u(t,\mx)) 
\left(\Div (\mff_{\mx}(u(t,\mx))-\Div(f(t,\xi))\Big|_{\xi=u(t,\mx)} \right) \\
&\qquad\qquad\qquad\qquad\qquad\qquad\qquad\qquad\qquad\qquad+\int_0^{u(t,\mx)} S'(\xi)\Div \mff_\mx'(\xi)\,d\xi \nonumber\\
&\overset{\eqref{geomcomp}}{=}S'(u(t,\mx))\Div(\mff_{\mx}(u(t,\mx))-S'(u(t,\mx))\Div \Div (A_{\mx}(\xi))\Big|_{\xi=u(t,\mx)} \nonumber \\&\qquad\qquad\qquad\qquad\qquad\qquad\qquad\qquad\qquad\qquad+\int_0^{u(t,\mx)}S'(\xi)\Div \Div A'_{\mx}(\xi)\,d\xi\\
&\overset{\eqref{f6}}{=} S'(u(t,\mx))\Div(\mff_{\mx}(u(t,\mx))+\Div \Div \Big(\int_0^{u(t,\mx)}A'_{\mx}(\xi)S'(\xi)\,d\xi\Big) \nonumber \\
&\qquad\qquad-S'(u(t,\mx))\Div\Div(A_{\mx}(u(t,\mx))-S''(u(t,\mx))\langle A'_{\mx}(u(t,\mx)) \nabla u,\nabla u \rangle,
\end{align*} which is  \eqref{a1}. \end{proof}

Another property of the entropy solution that we shall require is the so-called chain rule. 
It was introduced in \cite{CP} in the homogeneous case and adapted to the inhomogeneous situation in \cite{CK}. 
To formulate it, we first recall that $A_\mx'(\xi)=\sigma_\mx(\xi)^\top \sigma_\mx(\xi)$  by \eqref{Apd} and note that if $\sigma$ is locally given by $\sigma=\sigma^k_i \frac{\pa}{\pa x^k}\otimes dx^i$, then
\begin{equation}\label{transpose}
\sigma^\top = (\sigma^\top)^k_i \frac{\pa}{\pa x^k}\otimes dx^i \ \text{with} \ 
(\sigma^\top)^k_i = g^{kl}\sigma^m_l g_{mi}.
\end{equation}
Given $\psi:\R \to \R^+$, we now consider $\beta(\mx,\xi)$ such that $\beta'(\mx,\xi)= \sigma_\mx^\top(\xi)$, and
$\beta^\psi(\mx,\xi)$ such that $(\beta^\psi)'(\mx,\xi) = \sqrt{\psi(\xi)}\sigma_\mx^\top(\xi)$ and $\beta(\mx,0)=\beta^\psi(\mx,0)=0$
 (recall that a prime here denotes the derivative with respect to the real variable $\xi$). 
In local coordinates, this reads
\begin{align}
\label{beta}
\begin{split}
(\beta^{k}_{i})'(\mx,\xi) &= g^{kl}(\mx)\sigma^m_l(\mx,\xi) g_{mi}(\mx),\\
  ((\beta^\psi)^k_i)'(\mx,\xi) &= \sqrt{\psi(\xi)} g^{kl}(\mx)\sigma^m_l(\mx,\xi) g_{mi}(\mx).
\end{split}
\end{align}

We will need the following result on the divergences of the $\beta$-tensors:

\begin{theorem} (Chain rule)
\label{chain}
If $u:[0,\infty)\times M\to \R$ is a non-negative bounded $H^{1,2}(\R^+\times M)$ function, then for any non-negative $\psi\in C(\R)$  we have
\begin{align}
\label{crule}
\begin{split}
&\Div\Big(\beta^{ \psi}(\mx,u(t,\mx))\Big)-\Div(\beta^{ \psi}(\mx,\xi))\Big|_{\xi=u(t,\mx)}\\
&\qquad\qquad=\sqrt{\psi(u(t,\mx))} \Big(\Div(\beta(\mx,u(t,\mx))-\Div \beta(\mx,\xi )\Big|_{\xi=u(t,\mx)}\Big).
\end{split}
\end{align} 
\end{theorem}
\begin{proof} Using \eqref{div11}, and writing $u$ for $u(t,\mx)$ we calculate 
\begin{align}
\label{8}
\begin{split}
\Div (\beta^\psi(\mx,u))_i &=\pa_j  ((\beta^\psi)^j_i(\mx,u))
+ (\beta^\psi)^l_i(\mx,u) \Gamma_{jl}^j- (\beta^\psi)^j_l(\mx,u)\Gamma_{ji}^l \\
&  =\sqrt{\psi(u)} (\sigma^\top_\mx)_i^j(u) \pa_j u+\int_0^u\sqrt{\psi(\xi)} \pa_j 
 (\sigma^\top_\mx)_i^j(\xi)\, d\xi   \\
& \qquad+ (\beta^\psi)^l_i(\mx,u)\Gamma_{jl}^i  -(\beta^\psi)^j_l(\mx,u) \Gamma_{ji}^l .
\end{split}
\end{align} 
Also, 
\begin{equation*}
\Div (\beta^\psi(\mx,\xi))_i\Big|_{\xi=u}=\int_0^u\sqrt{\psi(\xi)}\pa_j 
(\sigma^\top_\mx)_i^j(\xi)\, d\xi+ (\beta^\psi)^l_i(\mx,u)\Gamma_{jl}^j
- (\beta^\psi)^j_l(\mx,u)\Gamma_{ji}^l,
\end{equation*} 
and therefore (compare with \eqref{8})
$$
\Div (\beta^\psi(\mx,u))_i-\Div (\beta^\psi(\mx,\xi))_{i}\Big|_{\xi=u}=\sqrt{\psi(u)} 
(\sigma^\top_\mx)_i^j(u) \pa_j u.
$$ Analogously,
\begin{equation}\label{divbs}
\Div (\beta(\mx,u))_i-\Div (\beta(\mx,\xi))_i\Big|_{\xi=u}= (\sigma^\top_\mx)_i^j(u) \pa_j u,
\end{equation} 
which gives the claim. 
\end{proof} 
Combining \eqref{divbs} with \eqref{transpose} we obtain that we can rewrite the last term in Theorem \ref{thm1} using:
\begin{equation}\label{adudu}
\langle A'_{\mx}(u) \nabla u,\nabla u \rangle = 
\left|\Div(\beta(\mx,u(t,\mx))-\Div(\beta(\mx,\xi )\big|_{\xi=u(t,\mx)} \right|_g^2,
\end{equation}
where $|\omega|_g = (g^{ij}\omega_i \omega_j)^{1/2}$ denotes the norm induced by $g$ on the space of one-forms.

Following \cite{CP}, we are next going to introduce an appropriate concept of entropy solution to \eqref{main-eq}, \eqref{ic}.
The definition of entropy solutions, as well as ultimately the proof of their existence, rests on vanishing viscosity approximations 
\begin{equation}
\label{VV}
\pa_t u_\eta +\Div \mff_{\mx}(u_\eta)=\Div(\Div ( A_{\mx}(u_\eta) ) )+\eta \Delta u_\eta,
\end{equation} 
where $\eta>0$ is some small constant. Here, $\Delta$ is the Laplace-Beltrami operator on the manifold given, 
for any $h\in C^\infty(M)$, by $\Delta h =\Div\circ \nabla h$, with $\Div$ and $\nabla$ 
as in \eqref{div-g} and \eqref{nabla}, 
respectively. In terms of local coordinates, setting $|g|:=\det g$, 
\begin{equation}\label{laplacedef}
\Delta h = \frac{1}{\sqrt{|g|}}\pa_i (\sqrt{|g|}g^{ij}\pa_j h).
\end{equation} 
It follows from Theorem \ref{thm1} and \eqref{adudu} that if $u_\eta$ is a bounded $H^{1,2}(\R^+\times M)$-solution to \eqref{VV} then
\begin{align}
\label{entr-ineq-prep}
\begin{split}
&\pa_t S(u_\eta)+\Div \int_0^{u_\eta(t,\mx)} \mff'_{\mx}(\xi)S'(\xi) d\xi = \Div \Div (\int_0^{u_\eta(t,\mx)} A_{\mx}'(\xi)S'(\xi) d\xi)\\
&-S''(u_\eta)\left|\Div(\beta(\mx,u_\eta(t,\mx))-\Div(\beta(\mx,\xi )\big|_{\xi=u_\eta(t,\mx)} \right|_g^2+\eta S'(u_\eta)\Delta u_\eta
\end{split}
\end{align}
Noting that we have $\Delta(S(w)) = S''(w) |\nabla w|^2 + S'(w)\Delta w$ for any bounded $H^2$-function $w$ on $M$, we can rewrite the
last term in \eqref{entr-ineq-prep} as 
$$
\eta S'(u_\eta)\Delta u_\eta = \eta \Delta S(u_\eta) - \eta S''(u_\eta) |\nabla u_\eta|^2.
$$
Finally, 
\begin{equation}
\label{m-eta}
\eta S''(u_\eta) |\nabla u_\eta|^2 = \int_{\R} S''(\xi)m_\eta(t,\mx,\xi)d\xi,
\end{equation}
where $m_\eta(t,\mx,\xi)=\eta \delta(\xi-u_\eta(t,\mx))|\nabla u_\eta|^2$ is a non-negative measure on $[0,\infty)\times M\times \R$. We shall also denote 
\begin{equation}
\label{n-eta}
n_\eta(t,\mx,\xi)=\delta(\xi-u_\eta(t,\mx))\left|\Div(\beta(\mx,u_\eta(t,\mx)))-\Div(\beta(\mx,\xi )) \right|_g^2,
\end{equation} 
which is a non-negative measure as well.  Note that for $u_\eta \geq 0$ the measures $n_\eta$ and $m_\eta$ are both supported in $[0,\infty)\times M\times [0,\infty)$. So we may rewrite \eqref{entr-ineq-prep} as 
\begin{align}
\label{new}
\begin{split}
\pa_t S(u_\eta)&+\Div \int_0^{u_\eta(t,\mx)} \mff'_{\mx}(\xi)S'(\xi) d\xi = \Div \Div \Big(\int_0^{u_\eta(t,\mx)} A_{\mx}'(\xi)S'(\xi) d\xi\Big)\\
&-\int_{0}^\infty S''(\xi)(n_\eta (t,\mx,\xi)+m_\eta (t,\mx,\xi))d\xi + \eta \Delta S(u_\eta)
\end{split}
\end{align}

Further, if we choose $S(u)=u^2/2$ and then integrate \eqref{new} over $M \times [0,\infty)$, we have 

\begin{equation}
\label{energy-m} 
\int_{\R^+\times M \times \R^+} (n_\eta + m_\eta ) d\xi d\mu(\mx) dt = \int_{M} \frac{1}{2}|u_0(\mx)|^2 d\mu(\mx) \, 
\end{equation}

Integration here is carried out with respect to the Riemannian density $\sqrt{|g|}$ induced
by $g$.   In local coordinates, $d\mu(\mx)=\sqrt{|g|}\,d\mx$, where $\sqrt{|g|} = \sqrt{|\det(g_{ij})|}$. 

Based on these observations, the following definition 
of entropy solutions extracts those properties that are stable under strong convergence (analogous
to \cite[Def.\ 2.1]{CP}).

\begin{definition}
\label{def-entropy}
We say that the measurable function $u:[0,\infty)\times M\to [0,1]$ is an entropy solution to \eqref{main-eq}, \eqref{ic} if 

\noindent(i)
\begin{equation}\label{L2-reg}
\Div(\beta(\mx,u(t,\mx))-\Div(\beta(\mx,\xi))\Big|_{\xi=u(t,\mx)} \in L^2([0,\infty)\times M); 
\end{equation}
(ii) There exists a non-negative measure $m$ on $[0,\infty)\times M\times [0,\infty)$ 
such that for any function $S\in C^2([0,\infty))$, 
the following equality
holds, together with the initial condition \eqref{ic}, in the sense of distributions on ${\cal D}'([0,\infty)\times M)$:
\begin{align}
\label{entr-ineq}
\begin{split}
&\pa_t S(u)+\Div \int_0^{u(t,\mx)} \mff'_{\mx}(\xi)S'(\xi) d\xi = \Div \Div \Big(\int_0^{u(t,\mx)} A_{\mx}'(\xi)S'(\xi) d\xi\Big)\\
&-S''(u)\left|\Div(\beta(\mx,u(t,\mx)))-\Div(\beta(\mx,\xi ))\big|_{\xi=u(t,\mx)} 
\right|_g^2 - \int_{0}^\infty S''(\xi)m(t,\mx,\xi)d\xi;
\end{split}
\end{align}
\noindent(iii) The chain rule \eqref{crule} holds in $L^2(\R^+\times M)$.
\end{definition} 

Equation \eqref{VV} is not a standard viscous approximation, but it is still a strictly parabolic equation. A viable approach to establishing
existence of entropy solutions to \eqref{main-eq}, \eqref{ic} would be to invoke \cite[Section V]{LSU}
to obtain existence of a solution to \eqref{VV}, \eqref{ic} for every $\eta>0$ and then showing that the net $(u_\eta)$ so obtained converges 
(in an appropriate sense) towards the entropy solution to \eqref{main-eq}, \eqref{ic}. Instead of implementing this approach directly, 
we shall first introduce a kinetic formulation of \eqref{main-eq}, \eqref{ic} on $M$ and then prove existence of the entropy 
solution by proving uniqueness of the kinetic solution (see e.g.\ \cite{BLf, Dpe, Pan_tams} for such an approach in the case of scalar conservation laws).

To this end, let us rewrite \eqref{entr-ineq} in the kinetic formulation. Set
$$
\chi_u(t,\mx,\xi) :=
\begin{cases}
\hphantom{-}1, & 0\leq \xi \leq u(t,\mx)\\
-1, & u(t,\mx) \leq \xi \leq 0\\
\hphantom{-}0, & \text{otherwise.}
\end{cases}
$$ 
Notice that if $0\leq u$, then for $\xi\geq 0$,
\begin{equation}
\label{positive}
\chi_u(t,\mx,\xi)={\rm sgn}_+(u(t,\mx)-\xi).
\end{equation} 
Taking into account that when $h(\mx,0)=0$, we have
\begin{equation}
\label{integralofchi}
h(\mx,u(t,\mx))=h(\mx,u(t,\mx))-h(\mx,0)=\int_{\R} h'(\mx, \xi) \chi_u(t,\mx,\xi) d\xi,
\end{equation}
we can rewrite \eqref{entr-ineq} in the so-called kinetic form as follows:
\begin{align*}
\pa_t \int_{\R}S'(\xi) \chi_u\, d\xi +\Div & \Big(\int_{\R} \chi_u S'(\xi) \mff'_{\mx}(\xi)\, d\xi\Big) \\
&= \Div \Div \Big(\int_{\R} \chi_u S'(\xi) A_{\mx}'(\xi) d\xi \Big)-\int_{\R} S''(\xi) (n+m)\, d\xi,
\end{align*} 
where 
\begin{align}\label{eq:def n}
n(t,\mx,\xi)&=\delta(\xi-u(t,\mx))\left|\Div(\beta(\mx,u(t,\mx)))-\Div(\beta(\mx,\xi )) \right|_g^2.
\end{align} 
Considering $S'$ as a test function supported in $(0,\infty)$, we conclude 
\begin{equation}
\label{kinetic}
\pa_t \chi_u +\Div (\chi_u\mff'_{\mx}(\xi)) = \Div \Div (\chi_u A_{\mx}'(\xi))+\pa_\xi (n+m), \ \ \xi \in (0,\infty).
\end{equation} 

 Next, we shall need a local version of the kinetic equation. 
Accordingly, let $\phi \in C_c^{2}(M)$. Then multiplying \eqref{kinetic} by $\phi$ and 
inserting gives
\begin{equation*}
\begin{split}
\pa_t (\phi\chi_u) +\Div \left(\phi \chi_u \mff'_{\mx}(\xi) \right) &= 
\phi \pa_t \chi_u + \phi \Div(\chi_u \mff'_{\mx}(\xi)) +\chi_u \mff'_{\mx}(\xi)(\phi)\\
&= \phi \Div \Div (\chi_u A_{\mx}'(\xi)) + \phi \pa_\xi (n+m) + \chi_u \mff'_{\mx}(\xi)(d\phi).
\end{split}
\end{equation*}
Furthermore,
\begin{equation*}
\begin{split}
\Div \Div (\phi \chi_u A_{\mx}'(\xi)) &= \Div(\chi_u A_{\mx}'(\xi)(d\phi))
+\Div(\phi \Div(\chi_u A_{\mx}'(\xi))) \\
&=\Div(\chi_u A_{\mx}'(\xi)(d\phi)) + \phi \Div \Div (\chi_u A_{\mx}'(\xi))
+ \Div(\chi_u A_{\mx}'(\xi))(\nabla\phi),
\end{split}
\end{equation*}
so that we arrive at
\begin{equation}
\begin{split}
\label{kinetic-loc}
\pa_t (\phi\chi_u)+\Div \left(\phi \chi_u \mff'_{\mx}(\xi) \right)=&\Div \Div (\phi \chi_u A_{\mx}'(\xi)  )+ \phi \pa_\xi (n+m)+\chi_u  \mff'_{\mx}(\xi)(d\phi) \\ 
&-\Div \left(\chi_u A'_{\mx}(\xi)(d\phi) \right)-\Div \left(\chi_u A'_{\mx}(\xi)\right)(\nabla \phi).
\end{split}
\end{equation} 
Our goal is to analyze \eqref{kinetic-loc} in local charts by regularization. To this end, we shall employ a
standard mollifier $\rho_{\eps,\delta}\in {\mathcal D}([0,\infty)\times\R^d\times[0,\infty))$ 
of the form (below and in the sequel, in order to avoid proliferation of symbols, we shall by a slight abuse of notation denote 
convolution kernels for $t$, $\mx$ or $\xi$ by the same letter)
\begin{equation}\label{eq:defrho}
\rho_{\eps,\delta}(t,\mx,\xi)=\frac{1}{\eps^{d+1}\delta} \omega_1\Big(\frac{\xi}{\delta}\Big) 
\omega_1\Big(\frac{t}{\eps}\Big) \prod_{j=1}^d \omega_2\Big(\frac{x_j}{\eps}\Big)=\rho_\eps(t,\mx)\rho_{\delta}(\xi),
\end{equation} where $\rho_\eps(t,\mx):=\frac{1}{\eps^{d+1}}  
\omega_1(\frac{t}{\eps}) \prod_{j=1}^d \omega_2(\frac{x_j}{\eps})$ and
$\rho_{\delta}(\xi):=\frac{1}{\delta}\omega_1(\frac{\xi}{\delta})$.

Here, $\omega_1$, $\omega_2\in {\mathcal D}(\R)$ are non-negative compactly supported smooth functions with total mass one,
and $\supp(\omega_1)\subseteq (-1,0)$. 
For a distribution $F\in {\cal D}'([0,\infty)\times \R^d \times [0,\infty))$ we set
\begin{equation}
\label{reg}
F^{\eps,\delta}=F\star \rho_{\eps,\delta}, \ \ F^{\eps}=F\star \rho_{\eps},
\end{equation} where
\begin{equation}
\label{reg-def}
\begin{cases}
F\star \rho_{\eps,\delta}(t,\mx,\xi)&=\langle F, \rho_{\eps,\delta}(t-\cdot,\mx-\cdot,\xi-\cdot) \rangle\\
F\star \rho_{\eps}(t,\mx,\xi)&=\langle F, \rho_{\eps}(t-\cdot,\mx-\cdot) \rangle.
\end{cases}
\end{equation} 
For a distribution $F\in {\cal D}'([0,\infty)\times M \times \R)$ compactly supported in a single chart $(U_\alpha,\psi_\alpha)$ 
we set (suppressing the dependence on $\alpha$ notationally) 
\begin{equation}
\label{loc-reg}
F^{\eps,\delta}=((F\circ\psi_\alpha^{-1})\star \rho_{\eps,\delta})\circ \psi_\alpha,
\end{equation} 
with $F\circ\psi_\alpha^{-1}$ the pullback of $F$ under the diffeomorphism $\psi_\alpha^{-1}$.
Finally, recalling that we work with non-negative solutions, which automatically provide \eqref{positive}, we introduce the following definition.
\begin{definition} 
\label{def-kinetic} Let $\mcf$ be a set of tuples $(\chi,u_0,m,n)$, containing the following data: 
$\chi:[0,\infty)\times M \times [0,\infty) \to [0,1]$ is measurable, $\xi\mapsto \chi(t,\mx,\xi)$ is compactly supported, uniformly in $(t,\mx)$, and is non-increasing
for $\xi\in [0,\infty)$. The function $u_0:M\to [0,1]$ is measurable, and $m, n \in {\cal M}([0,\infty)\times M \times [0,\infty))$ are 
non-negative Radon measures. 
Then $\mcf$ is called {\em kinetically admissible} if it satisfies:

\noindent(i) For any $(\chi,u_0,m,n)\in \mcf$, the following Cauchy problem is satisfied: 
\begin{align}
\label{kin-eq}
\pa_t \chi+\Div (\chi\mff'_{\mx}(\xi)) &= \Div \Div (\chi A_{\mx}'(\xi))+\pa_\xi (n+m)\\
\label{kin-ic}
\chi(0,\mx,\xi)&={\rm sgn}_+(u_0(\mx)-\xi)
\end{align} for $(t,\mx,\xi) \in [0,\infty)\times M \times [0,\infty)$. 

\noindent(ii) For any two tuples $(\chi,u_0,m,n), (\tilde\chi,v_0,\tilde m,\tilde n) \in {\cal F}$, there exist 
a finite atlas $(U_\alpha,\psi_\alpha)_{\alpha=1}^k$ for $M$
and non-negative smooth functions $\phi_\alpha\in {\mathcal D}(U_\alpha)$ such that $\phi_\alpha^2$ is a partition of unity and constants $C_\alpha$ such that the following estimate holds for a.e.\ $t$: 
\begin{equation}
\label{secondorder}
\begin{split}
& \limsup_{\eps\to 0}
\limsup_{\delta\to 0}\Big[\int_{[0,t)\times M\times[0,\infty)}(\phi_\alpha \pa_{\xi}(m+ n))^{\eps,\delta}  (\phi_\alpha(1-\tilde\chi))^{\eps,\delta} - \\
& \hphantom{x}\qquad\qquad\qquad \qquad\qquad \qquad\qquad (\phi_\alpha\, \pa_{\xi}(\tilde m +\tilde n))^{\eps,\delta}  (\phi_\alpha \chi)^{\eps,\delta} dt d\mu(\mx) d\xi \\
&\qquad\qquad+\!\int_{[0,t)\times M\times[0,\infty)} (\Div \Div (\phi_\alpha \chi A'))^{\eps,\delta}\,  (\phi_\alpha(1- \tilde\chi))^{\eps,\delta} dt d\mu(\mx) d\xi  \\ 
&\qquad\qquad +\!\int_{[0,t)\times M\times[0,\infty)}(\Div \Div (\phi_\alpha (1-\tilde\chi) A'))^{\eps,\delta}\,  (\phi_\alpha \chi)^{\eps,\delta} dt d\mu(\mx) d\xi \Big] \\
& \qquad\qquad\qquad\qquad\leq   C_\alpha \int_{[0,t)\times M\times[0,\infty)}\chi (1-\tilde\chi) dt d\mu(\mx) d\xi.
\end{split}
\end{equation}

A measurable function $\chi:[0,\infty)\times M \times [0,\infty) \to [0,1]$, $\xi\mapsto \chi(t,\mx,\xi)$ that is compactly supported with respect to $\xi$ uniformly in $(t,\mx)$, and is non-increasing with respect to $\xi\in [0,\infty)$, is called an ${\cal F}$-kinetic solution if there exist measures 
$m, n\in {\cal M}([0,\infty)\times M \times [0,\infty))$ such that $(\chi,u_0,m,n) \in {\cal F}$.
\end{definition} 
\begin{remark}\label{rem5}
(i) The initial value in \eqref{kin-ic} is understood to be attained in 
the weak sense, i.e., for any test function $\varphi$ we have
\begin{equation}\label{icaslimit}
\lim_{t\to 0} \int \chi (t,x,\xi) \varphi(x,\xi)d\mu(\mx) d\xi = \int {\rm sgn}_+(u_0(\mx)-\xi) \varphi(x,\xi)\,d\mu(\mx) d\xi.
\end{equation}
(ii) Since $\phi_\alpha$ is supported in a single chart, the regularizations in \eqref{secondorder} are defined as in \eqref{loc-reg}.

\noindent(iii)  As introduced after \eqref{energy-m}, $d\mu(\mx)$ denotes the Riemannian
density associated with the metric $g$. 
\end{remark}

Our approach differs from the kinetic solution concept from \cite{CP} since here we do not a priori impose the form of the kinetic function (i.e., we do not assume that it has the form $\frac{1}{2} \left({\rm sgn}(u-\xi)+{\rm sgn}(\xi) \right)$ as in \cite{CP}; see formula (2.15) there). It is also more typical in the theory to call kinetic solution a function depending on time, space and kinetic variables satisfying some additional properties (see e.g. \cite{D, Pan_tams}). Note that we can have several kinetically admissible sets. This is natural since there may exist several stable semigroups of solutions to \eqref{main-eq}, \eqref{ic} essentially depending on the approximation that we use (see e.g. \cite{AKR, Lfl} for conservation laws).

\section{Uniqueness}\label{uniquesec}
Our first goal in this section is to derive a uniqueness result for elements of a kinetically admissible set whose initial data coincide. To prove
this we will rely on the following version of Friedrichs' Lemma, which follows as in \cite[17.1.5]{H3}:
\begin{lemma}\label{friedrichs} Let $\varphi$ be a standard mollifier 
($\varphi\in {\mathcal D}(\R^d)$, $\int \varphi=1$).
Set $\varphi_\eps(x)=\eps^{-d}\varphi(x/\eps)$ and let $1\le j \le d$.
\begin{itemize}
\item[(i)] Let $h\in C^2(\R^d)$, $k\in L^\infty(\R^d)$. Then 
$$
\pa_j(hk)\star \varphi_\eps - \partial_j(h(k\star \varphi_\eps)) \to 0 \quad (\eps\to 0)\quad \mbox{in }\ L^1_{loc}(\R^d).
$$
\item[(ii)] Let $v\in L^1(\R^d)$ be compactly supported, and let $a\in C^1(\R^d)$. Then
$$
(a\partial_j v)\star \varphi_\eps - a (\partial_j v\star \varphi_\eps) \to 0 \quad (\eps\to 0)\quad \mbox{in }\ L^1(\R^d).
$$
\end{itemize}
\end{lemma}
\begin{theorem} \label{uniq-k} 
Assume that ${\cal F}$ is a kinetically admissible set for  \eqref{main-eq}, \eqref{ic}.  Then, for any two tuples $(\chi,u_0,m,n)$ and 
$(\tilde{\chi},v_0,\tilde m, \tilde{n})$ in $\mcf$, equality of $u_0$ and $v_0$ implies that $\chi=\tilde{\chi}$ and $\chi=\chi_u$ 
for a function $u\in L^\infty([0,\infty)\times M)$.
\end{theorem}
\begin{proof}
Assume that for the initial value $u_0$ with
$$
0 \leq u_0 \leq 1
$$ 
we have two tuples $(\chi, u_0, m, n)$ and $(\tilde{\chi}, u_0, \tilde m, \tilde{n})$ in $\mcf$.
Note that according to the geometry compatibility condition \eqref{geomcomp} we have 
\begin{align}
\label{kin-eq-1}
&\pa_t (1-\tilde{\chi}) +\Div ((1-\tilde{\chi})\mff'_{\mx}(\xi)) = 
\Div \Div ((1-\tilde{\chi}) A_{\mx}'(\xi)) -\pa_\xi (\tilde{n}+\tilde{m}).
\end{align} 
Now let $(\psi_\alpha,U_\alpha)_{\alpha=1}^k$ be a covering of $M$ by charts 
as in Def.\ \ref{def-kinetic}, with corresponding functions $\phi_\alpha$ ($\supp \phi_\alpha \Subset U_\alpha, \ 
\sum_{\alpha=1}^k \phi_\alpha^2 = 1$). Fixing $\alpha$,
we rewrite \eqref{kin-eq-1} in localized form for $\phi\equiv\phi_\alpha$. Then 
from \eqref{kinetic-loc} we obtain 
\begin{equation}
\begin{split}
\label{kinetic-loc-new}
\pa_t (\phi\chi)= &-\Div \left(\phi \chi \mff'_{\mx}(\xi) \right) + 
\Div \Div (\phi \chi A_{\mx}'(\xi))+
\phi \pa_\xi (n+m)+\chi  \mff'_{\mx}(\xi)(d\phi) \\ 
&-\Div \left(\chi A'_{\mx}(\xi)(d\phi) \right)-\Div \left(\chi A'_{\mx}(\xi)\right)(\nabla \phi),
\end{split}
\end{equation} and starting from \eqref{kin-eq-1} instead of \eqref{kinetic}, 
the proof of \eqref{kinetic-loc} shows that  
\begin{equation}  \label{kin-eq-1-loc} 
\begin{split}
\pa_t (\phi(1-\tilde{\chi})) =&-\Div \left(\phi (1-\tilde{\chi}) \mff'_{\mx}(\xi) \right)+\Div \Div (\phi (1-\tilde{\chi}) A_{\mx}'(\xi) ) \\
&  -\phi \pa_\xi (\tilde n+ \tilde m) 
+(1-\tilde{\chi})  \mff'_{\mx}(\xi)(d\phi)-\Div \left((1-\tilde{\chi}) A'_{\mx}(\xi)(d\phi) \right)\\
&-\Div \left((1-\tilde{\chi}) A'_{\mx}(\xi)\right)(\nabla \phi). 
\end{split}
\end{equation} 
Note that all terms in both \eqref{kinetic-loc-new} and \eqref{kin-eq-1-loc} are supported in a 
single chart $(\psi_\alpha,U_\alpha)$, so using push-forward under the chart map 
$\psi_\alpha$, we obtain an equivalent system of equations, this time on $\psi_\alpha(U_\alpha)
\subseteq \R^d$, 
and all differential operators occuring in \eqref{kinetic-loc-new} and \eqref{kin-eq-1-loc}
are transformed into the corresponding ones on $\R^d$ with respect to the push-forward
metric $(\psi_\alpha)_*g$. Moreover, all tensors and functions involved have compact support 
within $\psi_\alpha(U_\alpha)$, hence can be extended by $0$ to all of $\R^d$. 
Altogether, this means that we may assume, without loss of generality, that $M=\R^d$ and $g=g_{ij}$ is a Riemannian metric on $\R^d$.   

Now we convolve equations \eqref{kinetic-loc-new} and 
\eqref{kin-eq-1-loc} by $\rho_{\eps,\delta}$ and multiply them by $ (\phi (1-\tilde{\chi}))^{\eps,\delta}$ and $ (\phi \chi)^{\eps,\delta}$, 
respectively. Next we sum the equations so obtained and integrate over $[0,t)\times \R^{d+1}$. Then we find that the left hand side, 
$$
2\int_0^t \pa_t  \int_{\R^{d} }\int_0^\infty(\phi \chi)^{\eps,\delta} (\phi(1-\tilde{\chi}))^{\eps,\delta} d\xi d\mu(\mx) dt,
$$
can be written as a sum of six terms, which we treat separately. 
For the limiting behavior of the first term, we obtain 
\begin{equation}
\label{term1}
\begin{split}
\lim_{\eps\to 0}\lim_{\delta \to 0} \bigg(  &\int_{[0,t)\times \R^{d} \times[0,\infty)}- (\Div (\phi (1-\tilde{\chi}) \mff'_{\mx}(\xi)))^{\eps,\delta} 
 (\phi \chi)^{\eps,\delta}   \\ 
&-(\Div (\phi \chi\mff'_{\mx}(\xi)))^{\eps,\delta}  (\phi(1-\tilde{\chi}))^{\eps,\delta}\,dt d\xi d\mu(\mx) \bigg) \\
& \hphantom{xx}  = -\int_{[0,t)\times \R^{d} \times [0,\infty) }\Div (\mff'_{\mx}(\xi))\phi^2(1-\tilde{\chi})\chi\, dt d\xi d\mu(\mx) \\
& \hphantom{xx} \leq C \int_{[0,t)\times \R^{d} \times [0,\infty)} (1-\tilde{\chi})\chi \, dt d\xi d\mu(\mx),
\end{split}
\end{equation}
where we used the product rule and Lemma \ref{friedrichs} (ii) on one, and integration by parts on the other term. Similarly, for some bounded function $G$,
\begin{equation}
\label{term2}
\begin{split}
\lim_{\eps \to 0}\lim_{\delta \to 0} & \bigg( \int_{[0,t)\times \R^{d} \times [0,\infty)} -\Div \left((1-\tilde{\chi})  A'(d\phi) \right)^{\eps,\delta}(\phi \chi)^{\eps,\delta}\\
&\qquad\qquad\qquad\qquad-\Div \left(\chi A'(d\phi) \right)^{\eps,\delta}  (\phi (1-\tilde{\chi}))^{\eps,\delta} dt d\xi d\mu(\mx) \bigg) \\ 
= \lim_{\eps \to 0}\lim_{\delta \to 0} &\bigg(\int_{[0,t)\times \R^{d} \times [0,\infty)} -g^{ik} \pa_i \left((1-\tilde{\chi}) (A')^j_k \pa_j\phi \right)^{\eps,\delta}(\phi \chi)^{\eps,\delta} \hfill
\\ &\qquad\qquad \qquad \qquad -g^{ik} \pa_i \left(\chi (A')^j_k \pa_j\phi \right)^{\eps,\delta}(\phi(1-\tilde{\chi}))^{\eps,\delta} dt d\xi d\mu(\mx) \bigg) \\
&\qquad\qquad\qquad\qquad\qquad+\int_{[0,t)\times \R^{d} \times [0,\infty) }  G(\mx,\xi) (1-\tilde{\chi})\chi dt d\xi d\mu(\mx)\\
&\leq C \int_{[0,t)\times \R^{d} \times [0,\infty) } (1-\tilde{\chi})\chi dt d\xi d\mu(\mx),
\end{split}
\end{equation} 
and so on for all other terms except the ones involving $\Div \Div (\phi (1-\tilde{\chi})A_{\mx}'(\xi))-\phi \pa_\xi (\tilde n+ \tilde m)$ where we 
directly use \eqref{secondorder} to get the desired estimate. 

Thus we arrive at
\begin{equation}
\label{g-0}
\begin{split}
\lim_{\eps \to 0}\lim_{\delta \to 0} \bigg( \int_{\R^{d} }\int_0^\infty(\phi \chi)^{\eps,\delta} & (\phi(1-\tilde{\chi}))^{\eps,\delta} d\xi d\mu(\mx)\\
& -\Big( \int_{\R^{d} }\int_0^\infty(\phi \chi)^{\eps,\delta} (\phi(1-\tilde{\chi}))^{\eps,\delta} d\xi d\mu(\mx)\Big) \bigg|_{t=0} \bigg) \\ 
& \qquad\qquad \leq  C \int_{\R^{d}} \int_0^\infty \chi (1-\tilde{\chi}) d\xi d\mu(\mx).
\end{split}
\end{equation}  
The initial condition \eqref{icaslimit} implies that
$$
\lim_{\eps \to 0}\lim_{\delta \to 0} \bigg( \Big( \int_{\R^{d} } \int_0^\infty (\phi \chi)^{\eps,\delta} (\phi(1-\tilde{\chi}))^{\eps,\delta} d\xi d\mu(\mx)\Big) \bigg|_{t=0} \bigg) =0.
$$ 
Thus, rewriting \eqref{g-0} as an expression on $M$ results in
\begin{align}
\label{g-1}
\int_{M\times [0,\infty) }\phi_\alpha^2 \chi (1-\tilde{\chi}) d\xi d\mu(\mx)  \leq \tilde{C}_\alpha \int_0^t\int_{M\times [0,\infty)} \chi (1-\tilde{\chi}) d\xi d\mu(\mx).
\end{align}  
Now summing over $\alpha=1,\dots,k$ and setting $C:=\sum_{\alpha=1}^k\tilde{C}_\alpha <\infty$ gives
\begin{align}
\label{g-2}
\int_{M\times [0,\infty)}\chi (1-\tilde{\chi}) d\xi d\mu(\mx)  \leq C \int_0^t\int_{M\times [0,\infty)} \chi (1-\tilde{\chi}) d\xi d\mu(\mx).
\end{align}

From here, according to the Gronwall inequality, we conclude that
\begin{align*}
\chi(t,\mx,\xi)\, (1-\tilde{\chi}(t,\mx,\xi))=0
\end{align*}
for almost every $(t,\mx,\xi)\in [0,\infty)\times M \times \R$.

This implies that either $\chi(t,\mx,\xi)=0$ or $\tilde{\chi}(t,\mx,\xi)=1$.  Since we can interchange the roles of $\chi$ and $\tilde{\chi} $, we conclude that $1$ and $0$ are actually the only values that $\chi$ or $\tilde{\chi} $ can attain and that $\chi=\tilde{\chi}$. Since $\chi$ is also non-increasing with respect to $\xi$ on $[0,\infty)$, 
we conclude that there exists a function $u:[0,\infty)\times M \to \R$ such that 
\begin{equation}
\label{***} 
\chi(t,\mx,\xi)={\rm sgn}_+(u(t,\mx)-\xi).
\end{equation} 
In fact, this function is given by $u(t,\mx)= \int_0^\infty \chi(t,\mx,\xi)\, d\xi$. Note that this in particular shows that if $ \chi=\tilde{\chi}$ a.e.\ then
$$
u=\tilde{u}
$$
almost everywhere. 
\end{proof} 
Notice that from the proof of the previous theorem we see that every $\chi$ appearing in any tuple from ${\cal F}$ has the form ${\rm sgn}_+(u(t,\mx)-\xi)$ where the function $u$ satisfies \eqref{entr-ineq}.  

\section{Existence}\label{existencesec}

Our next aim is to prove that given initial data $u_0: M\to [0,1]$, 
 there exists a kinetic function $\chi$ and corresponding measures $m,n$ such that the conditions from Definition \ref{def-kinetic} are satisfied. To this end, consider the vanishing viscosity approximation \eqref{VV} augmented with the initial conditions \eqref{ic}. We have the following theorem.
 
\begin{theorem}
\label{t1} 
For any $\eta >0$ the initial value problem \eqref{VV}, \eqref{ic} has a unique solution $u_\eta $ in $H^{1,2}((0,\infty)$ $\times M)$ $\cap L^\infty((0,\infty)\times M)$. 
This solution satisfies, for any convex $S\in C^2(\R)$ with $S(0)=0$,

\begin{equation}
\label{app-ec}
\begin{split}
&\pa_t S(u_\eta)+\Div \int_0^{u_\eta(t,\mx)} \mff'_{\mx}(\xi)S'(\xi) d\xi = \Div \Div (\int_0^{u_\eta(t,\mx)} A_{\mx}'(\xi)S'(\xi) d\xi)\\
&-S''(u_\eta) \left|\Div(\beta(\mx,u_\eta(t,\mx)))-\Div(\beta(\mx,\zeta))\Big|_{\zeta=u_\eta(t,\mx)} \right|_g^2 \\ &+\eta \Delta S(u_\eta)-\eta S''(u_\eta)|\nabla u_\eta|^2.
\end{split}
\end{equation}

\end{theorem} 
\begin{proof}
Existence follows from \cite[Th.\ 15 and Rem.\ 16]{GKM}, while \eqref{app-ec} is a direct consequence of \eqref{entr-ineq-prep}. 
\end{proof}
 
We now want to prove that for such solutions $u_\eta$, the corresponding $\chi_{u_\eta}$, $n_\eta$ and $m_\eta$ defined through \eqref{entr-ineq-prep}, \eqref{n-eta}, and \eqref{m-eta}, converge to the function $\chi_u$ and measures $m$ and $n$ such that the set of all such limits $(\chi_u,u_0,n,m)$ is a kinetically admissible set in the sense of Definition \ref{def-kinetic}. Before we show convergence we will establish that there exist convergent subsequences such that their limits satisfy \eqref{secondorder} from Definition \ref{def-kinetic}. 
\begin{lemma} \label{subseqconverge} Let $u_\eta$ be a solution to \eqref{VV} with the initial data $u_\eta|_{t=0}=u_0$ and measures $n_\eta, m_\eta $. Then there exists a subsequence $\eta_n$ along which $\chi_{u_{\eta_n}}$ converges (in the weak-$\star $ topology) to some $\chi_u \in L^\infty([0,\infty)\times M 
\times [0,\infty))$ and such that the corresponding measures $n_{\eta_n}$ and $m_{\eta_n}$ converge weakly to Radon measures $n_u,m_u$. Furthermore such limits satisfy \eqref{kin-eq}, \eqref{kin-ic}.
\end{lemma}
\begin{proof}According to \eqref{energy-m}, we see that the sets $\{ n_{\eta} \}_{\eta>0}$ and $\{m_{\eta}\} _{\eta>0}$ are bounded in the space of Radon measures ${\cal M}([0,\infty)\times M\times [0,\infty))$. Also the $\chi_{u_\eta}$ are bounded between zero and one. Thus, we can find common weakly converging subsequences (see \cite[Theorem 1.1.2 and 1.1.4]{Evans}).
Equation \eqref{kin-eq} follows from rewriting \eqref{app-ec} in terms of $\chi_{u_\eta}$ (see \eqref{integralofchi} onwards) and letting $\eta \to 0$ (note that $\eta \Delta S(u_\eta)\to 0$). Now multiplying \eqref{kin-eq} by kink functions $f_j$ converging to ${\rm sgn}_+(T-t)$ and a test function $\varphi(\mx,\xi)$, integrating over all variables and letting first $n\to \infty$ and then $j\to \infty$ shows that the function $T\mapsto \int \chi (T,x,\xi) \varphi(x,\xi)d\mu(\mx) d\xi$ appearing in \eqref{icaslimit} is almost everywhere equal to a continuous function in $T$. This gives the initial condition \eqref{kin-ic}.  \end{proof} 

We will now show that the set of all limits of such subsequences satisfies the conditions of Definition \ref{def-kinetic}. To this end, let us first prove the following lemma. Since \eqref{secondorder} only deals with expressions of the form $\phi_\alpha \chi_u$ where $\phi_\alpha$ is compactly supported in a chart domain we may assume $M=\R^d$. Let us put 
\begin{equation}\label{bartilde}
a=A'_\mx, \ \ \bar{\chi}_v=\phi_\alpha(1-\chi_v), \ \ \tilde{\chi}_u=\phi_\alpha \chi_u.
\end{equation} 
Notice that for every fixed $\eps,\delta>0$ we have for every $(t,\mx,\xi)$ along the previously chosen subsequence 
\begin{equation}
\label{ae-conv}
\lim\limits_{n\to \infty} \chi_{u_n}^{\eps,\delta}(t,\mx,\xi)=\chi_{u}^{\eps,\delta}(t,\mx,\xi).
\end{equation} The same holds for $\bar{\chi}_{v_n}^{\eps,\delta}$ and $\tilde{\chi}_{u_n}^{\eps,\delta}$, as well as all their (partial) derivatives.
  
 Since $g$ is symmetric and positive definite there exists a symmetric square root (depending smoothly on the point) which we will denote by $h$, i.e., 
\[
g^{ij}=\delta_{kl}h^{il}h^{jk}
\] 
(where $\delta_{lk}$ is the Kronecker-Delta). 
\begin{lemma}
\label{lemma1}
 There exists a bounded function $G$ (depending on the metric, $a$ and $\phi_\alpha$, but not on $\eps,\delta$, or $n$) defined on 
$[0,\infty)\times \R^d \times [0,\infty)$
such that 
\begin{equation}\label{est-L1}
\begin{split}
&\!\!\!\!\!\!
\int_{\mathbb{R}^{d}} \int_0^\infty \Div\Div(\bar{\chi}_{v_n}a)\star\rho_{\eps,\delta}\tilde{\chi}_{u_n}^{\eps,\delta}d\mu 
d\xi+\int_{\mathbb{R}^{d}} \int_0^\infty \Div\Div(\tilde{\chi}_{u_n} a)\star\rho_{\eps,\delta}\bar{\chi}_{v_n}^{\eps,\delta} d\mu d\xi\\ 
&\!\!\!\!
-2\int_{\mathbb{R}^{3d}} \int_{(\mathbb{R}^+)^3}  \delta_{lm} h^{mi}(\my)h^{rl}(\my')\phi_\alpha(\my)\phi_\alpha(\my') \rho_{\eps}(t-\tau,\mx-\my)\rho_{\eps}(t-\tau',\mx-\my')\times \\ 
&\!\!\!\!
\times \left(\Div_{\my'} \left(\beta^{\rho^2_{\delta}(\xi-.)}(\my',v_n(\tau',\my'))\right)-
\left(\Div_{\my'}\left(\beta^{\rho^2_{\delta}(\xi-.)}(\my',\zeta)\right)\right)|_{\zeta=v_n(\tau',\my')}\right)_{r}\times\\
&\!\!\!\!
\qquad\qquad \times\left(\Div_{\my}\left(\beta^{\rho^2_{\delta}(\xi-.)}(\my,u_n(\tau,\my))\right) \right.\\
& \qquad\qquad\qquad\qquad
-\left.\left(\Div_{\my}\left(\beta^{\rho^2_{\delta}(\xi-.)}(\my,\zeta)\right)\right)|_{\zeta=u_n(\tau,\my)}\right)_{i}d\my d\tau d\my'd\tau' d\mu d\xi \\ 
&\qquad\qquad\qquad\qquad\qquad\qquad\approx \int_{\R^{d}} \int_0^\infty G(t,\mx,\xi) (1-\chi^{\eps,\delta}_{v_n}) \chi^{\eps,\delta}_{u_n} d\mu d\xi,
\end{split}
\end{equation} 
on $\left(\mathbb{R}^{d},g\right)$, 
where $\approx$ means that the difference of the left hand side and the right hand side goes to zero in $L_{\mathrm{loc}}^1(\mathbb{R}^+)$
(as a function of $t$) as, first, $n\to \infty$, second
$\delta\to 0$, and finally $\eps\to0$.
\end{lemma}
\begin{proof} 

Since the calculations required for this proof are quite extensive, we only summarize the main
steps here and outsource several arguments to the appendix. 
Also, to reduce the notational burden, we will suppress all $t$-dependencies: the $\tau$- and $\tau'$-integrations
remain untouched by the arguments used in the proof below, so we state all the required steps as if $u_n$, $v_n$, $\rho_\eps$, $G$, \dots
were independent of $t$, noting that re-introducing the actual dependencies then is entirely straightforward. Moreover, 
$\int_\R d\xi$ will always be understood to mean $\int_{\R^+} d\xi$.

In the computations below, we shall rely heavily on the Friedrichs lemma for convolutions
(cf.\ Lemma \ref{friedrichs}). To begin with, note that 
for any $f\in C^2(\R^{d+1})$ and any fixed $\eps,\delta$ 
\begin{multline*}
\lim_{n\to \infty}\int_{\mathbb{R}^{d+1}}\pa_{j}\bar{\chi}_{v_n}^{\eps,\delta}(\mx,\xi)\left(\left(\tilde{\chi}_{u_n}f\right)\star\rho_{\eps,\delta}\right)(\mx,\xi)d\mu(\mx)d\xi =\\ 
\int_{\mathbb{R}^{d+1}}\pa_{j}\bar{\chi}_{v}^{\eps,\delta}(\mx,\xi)\left(\left(\tilde{\chi}_{u}f\right)\star\rho_{\eps,\delta}\right)(\mx,\xi)d\mu(\mx)d\xi.
\end{multline*}
This is due to dominated convergence since $|\tilde{\chi}_{u_n}^{\eps,\delta}|\leq 1$, $|\pa_{j}\bar{\chi}_{v_n}^{\eps,\delta}|\leq \|\pa_{j}\rho_{\eps,\delta}\|_{L^1(\R^{d+1})} \leq C$ and both $\tilde{\chi}_{u_n}^{\eps,\delta}$ and $\bar{\chi}_{v_n}^{\eps,\delta}$ are supported in a compact set (which is independent of $n$). The same holds true for all integral expressions of similar form.

Therefore, whenever the difference of two such expressions (containing $\bar{\chi}_{v}^{\eps,\delta}$ and $\tilde{\chi}_u^{\eps,\delta}$) converges to zero  due to a variant of the Friedrichs lemma, the difference of the same expressions (only now containing $\bar{\chi}_{v_n}^{\eps,\delta}$ and $\tilde{\chi}_{u_n}^{\eps,\delta}$) converges to zero if we first let $n\to \infty$ and then 
$\delta,\eps \to 0$. So they are going to be equivalent for the limit (we use $\approx$ in our notation).

First, by \eqref{divdivstart} we obtain: 
\begin{equation}\label{eq:num0}
\begin{split}
&\int_{\mathbb{R}^{d+1}}\Div\Div(\tilde{\chi}_{u_n}a)\star\rho_{\eps,\delta}\bar{\chi}_{v_n}^{\eps,\delta}d\mu d\xi\\ 
& \qquad\qquad\qquad\qquad\qquad \approx -\int_{\mathbb{R}^{d+1}}\left(g^{ij}\Div(\tilde{\chi}_{u_n}a)_{i}\right)
\star\rho_{\eps,\delta}\pa_{j}\bar{\chi}_{v_n}^{\eps,\delta}d\mu d\xi.
\end{split}
\end{equation} 
That we do not have an actual equality here is merely due to the fact that some of the appearing Christoffel terms will be inside a convolution on one side of the equation but outside on the other. As outlined above, however, this does not cause a problem in the limit thanks to the Friedrichs lemma.

We continue with the right hand side of (\ref{eq:num0}). Expanding the remaining divergence, 
and using $\pa_kg^{ij}=-\Gamma^j_{ka}g^{ia}-\Gamma^i_{kb}g^{jb}$ and 
 $g^{ij}a^k_i=g^{ri}(\sigma^{T})_{r}^{k} (\sigma^{T})_{i}^{j}$, we find (see \eqref{a3new}):
\begin{equation}\label{eq:num1}
\begin{split}
\int_{\mathbb{R}^{d+1}} & \left(g^{ij} \Div(\tilde{\chi}_{u_n}a)_{i}\right)\star\rho_{\eps,\delta}\pa_{j}\bar{\chi}_{v_n}^{\eps,\delta}d\mu d\xi\\ 
&\approx
\int_{\mathbb{R}^{3d+3}}g^{ri}(\my)(\sigma^{T})_{r}^{k}(\my,\eta)(\sigma^{T})_{i}^{j}(\my,\eta)\bar{\chi}_{v_n}(\mz,\zeta)\tilde{\chi}_{u_n}(\my,\eta)\times \\
&\qquad\qquad \pa_{k}\rho_{\eps,\delta}(\mx-\my,\xi-\eta)\pa_{j}\rho_{\eps,\delta}(\mx-\mz,\xi-\zeta)d\my d\eta d\mz d\zeta d\mu(\mx)d\xi\\
&
+\int_{\mathbb{R}^{d+1}}\pa_{j}\bar{\chi}_{v_n}^{\eps,\delta}(\mx,\xi)\tilde{\chi}_{u_n}^{\eps,\delta}(\mx,\xi)\left[g^{ij}\Gamma_{ml}^{m}a_{i}^{l}
+\Gamma_{ka}^{j}g^{ia}a_{i}^{k}\right](\mx,\xi)d\mu(\mx)d\xi, 
\end{split}
\end{equation}
where the $\approx$ again stems from a variant of the Friedrichs lemma.

This allows us to calculate
\begin{equation}
\begin{split}
&\int_{\mathbb{R}^{d+1}}\left(g^{ij}\Div(\tilde{\chi}_{u_n}a)_{i}\right)\star\rho_{\eps,\delta}\pa_{j}\bar{\chi}_{v_n}^{\eps,\delta}d\mu d\xi\\
& \qquad\qquad\qquad\qquad\qquad\qquad+\int_{\mathbb{R}^{d+1}}\left(g^{ij}\Div(\bar{\chi}_{v_n}a)_{i}\right)\star\rho_{\eps,\delta}\pa_{j}\tilde{\chi}_{u_n}^{\eps,\delta}d\mu d\xi  \\
&\approx
\int_{\mathbb{R}^{3d+3}} \delta_{ml}h^{rl}(\my)(\sigma^{T})_{r}^{k}(\my,\eta)h^{mi}(\mz)(\sigma^{T})_{i}^{j}(\mz,\zeta) \times \\ 
&\bar{\chi}_{v_n}(\mz,\zeta)\tilde{\chi}_{u_n}(\my,\eta) \pa_{k}\rho_{\eps,\delta}(\mx-\my,\xi-\eta)\pa_{j}
\rho_{\eps,\delta}(\mx-\mz,\xi-\zeta)d\my d\eta d\mz d\zeta d\mu(\mx)d\xi\\ 
&+\int_{\mathbb{R}^{3d+3}} \delta_{ml}h^{rl}(\mz)(\sigma^{T})_{r}^{k}(\mz,\zeta)h^{mi}(\my)(\sigma^{T})_{i}^{j}(\my,\eta)\times \\ 
&\bar{\chi}_{v_n}(\mz,\zeta)\tilde{\chi}_{u_n}(\my,\eta) \pa_{k}\rho_{\eps,\delta}(\mx-\mz,\xi-\zeta)\pa_{j}
\rho_{\eps,\delta}(\mx-\my,\xi-\eta)d\my d\eta d\mz d\zeta d\mu(\mx)d\xi\\ 
& +
\int_{\mathbb{R}^{3d+3}} \delta_{ml}h^{rl}(\my)(\sigma^{T})_{r}^{k}(\my,\eta)\left[h^{mi}(\my)(\sigma^{T})_{i}^{j}(\my,\eta)
-h^{mi}(\mz)(\sigma^{T})_{i}^{j}(\mz,\zeta)\right]\times \\ 
&\bar{\chi}_{v_n}(\mz,\zeta)\tilde{\chi}_{u_n}(\my,\eta) \pa_{k}\rho_{\eps,\delta}(\mx-\my,\xi-\eta)\pa_{j}
\rho_{\eps,\delta}(\mx-\mz,\xi-\zeta)d\my d\eta d\mz d\zeta d\mu(\mx)d\xi \\ 
& +
\int_{\mathbb{R}^{3d+3}} \delta_{ml}h^{rl}(\mz)(\sigma^{T})_{r}^{k}(\mz,\zeta)
\left[h^{mi}(\mz)(\sigma^{T})_{i}^{j}(\mz,\zeta)-h^{mi}(\my)(\sigma^{T})_{i}^{j}(\my,\eta)\right]\times \\ 
& \bar{\chi}_{v_n}(\mz,\zeta)\tilde{\chi}_{u_n}(\my,\eta) \pa_{k}\rho_{\eps,\delta}(\mx-\mz,\xi-\zeta)\pa_{j}
\rho_{\eps,\delta}(\mx-\my,\xi-\eta)d\my d\eta d\mz d\zeta d\mu(\mx)d\xi \\ 
& +\int_{\mathbb{R}^{d+1}}\pa_{j}\bar{\chi}_{v_n}^{\eps,\delta}\tilde{\chi}_{u_n}^{\eps,\delta}
\left[g^{ij}\Gamma_{ml}^{m}a_{i}^{l}+\Gamma_{ka}^{j}g^{ia}a_{i}^{k}\right]d\mu d\xi\\
&+\int_{\mathbb{R}^{d+1}}\bar{\chi}_{v_n}^{\eps,\delta}\pa_{j}\tilde{\chi}_{u_n}^{\eps,\delta}
\left[g^{ij}\Gamma_{ml}^{m}a_{i}^{l}+\Gamma_{ka}^{j}g^{ia}a_{i}^{k}\right]d\mu d\xi. \label{eq:num2}
\end{split}
\end{equation}
Looking at the fourth term from (\ref{eq:num2}) another lengthy calculation and invocation of the Friedrichs lemma (see \eqref{a4new}) gives
\begin{equation}
\begin{split}
&\int_{\mathbb{R}^{3d+3}} \delta_{ml}h^{rl}(\mz)(\sigma^{T})_{r}^{k}(\mz,\zeta)\left[h^{mi}(\mz)(\sigma^{T})_{i}^{j}(\mz,\zeta)
-h^{mi}(\my)(\sigma^{T})_{i}^{j}(\my,\eta)\right]\times \\
&\bar{\chi}_{v_n}(\mz,\zeta)\tilde{\chi}_{u_n}(\my,\eta)  \pa_{k}\rho_{\eps,\delta}(
\mx-\mz,\xi-\zeta)\pa_{j}\rho_{\eps,\delta}(\mx-\my,\xi-\eta)d\my d\eta d\mz d\zeta d\mu(\mx)d\xi \\ 
&\approx
\int_{\mathbb{R}^{3d+3}}  \delta_{ml}h^{rl}(\mz)(\sigma^{T})_{r}^{k}(\mz,\zeta)\left[h^{mi}(\mz)(\sigma^{T})_{i}^{j}(\mz,\zeta)-h^{mi}(\my)(\sigma^{T})_{i}^{j}(\my,\eta)\right]\times \\ 
&\bar{\chi}_{v_n}(\mz,\zeta)\tilde{\chi}_{u_n}(\my,\eta) \pa_{j}\rho_{\eps,\delta}(\mx-\mz,\xi-\zeta) \pa_{k}\rho_{\eps,\delta}(\mx-\my,\xi-\eta)d\my d\eta d\mz d\zeta d\mu(\mx)d\xi \\ 
& +
\int_{\mathbb{R}^{d+1}} \delta_{ml}h^{rl}\bar{\chi}^{\eps,\delta}_{v_n}\tilde{\chi}^{\eps,\delta}_{u_n}(\sigma^{T})_{r}^{k}
\Gamma_{ks}^{s}\pa_{j}(h^{mi}(\sigma^{T})_{i}^{j})d\mu d\xi\\
&\qquad\qquad\qquad\qquad\qquad-\int_{\mathbb{R}^{d+1}} \delta_{ml}h^{rl}\bar{\chi}^{\eps,\delta}_{v_n}\tilde{\chi}^{\eps,\delta}_{u_n}(\sigma^{T})_{r}^{k}
\Gamma_{js}^{s}\pa_{k}(h^{mi}(\sigma^{T})_{i}^{j})d\mu d\xi.
\label{eq:num2.5}
\end{split}
\end{equation}
The last two terms in the equation above are of the form 
\begin{equation}\label{eq:62withbartilde} \int\bar{\chi}^{\eps,\delta}_{v_n}\tilde{\chi}^{\eps,\delta}_{u_n} G(\mx,\xi) d\mu d\xi
\end{equation}
for an appropriate function $G$ (which is bounded and independent of $n$). 
 By \eqref{bartilde} and Lemma \ref{friedrichs} it follows that 
the difference of this expression and the right hand side of \eqref{est-L1} (with the
functions $G$ only differing by a factor of $\phi_\alpha^2$) is $\approx 0$.

So the third and fourth term from (\ref{eq:num2}) together give 
\begin{equation}
\begin{split}
&\int_{\mathbb{R}^{3d+3}} \delta_{ml}h^{rl}(\my)(\sigma^{T})_{r}^{k}(\my,\eta)
\left[h^{mi}(\my)(\sigma^{T})_{i}^{j}(\my,\eta)-h^{mi}(\mz)(\sigma^{T})_{i}^{j}(\mz,\zeta)\right]\times \\ 
&\bar{\chi}_{v_n}(\mz,\zeta)\tilde{\chi}_{u_n}(\my,\eta) \pa_{k}\rho_{\eps,\delta}(\mx-\my,\xi-\eta)\pa_{j}
\rho_{\eps,\delta}(\mx-\mz,\xi-\zeta)d\my d\eta d\mz d\zeta d\mu(\mx)d\xi \\ 
&+
\int_{\mathbb{R}^{3d+3}} \delta_{ml}h^{rl}(\mz)(\sigma^{T})_{r}^{k}(\mz,\zeta)\left[h^{mi}(\mz)(\sigma^{T})_{i}^{j}(\mz,\zeta)
-h^{mi}(\my)(\sigma^{T})_{i}^{j}(\my,\eta)\right]\times \\ 
&\bar{\chi}_{v_n}(\mz,\zeta)\tilde{\chi}_{u_n}(\my,\eta) \pa_{k}\rho_{\eps,\delta}(\mx-\mz,\xi-\zeta)
\pa_{j}\rho_{\eps,\delta}(\mx-\my,\xi-\eta)d\my d\eta d\mz d\zeta d\mu(\mx)d\xi \\ 
&\approx
\int_{\mathbb{R}^{3d+3}} \delta_{ml}h^{rl}(\my)(\sigma^{T})_{r}^{k}(\my,\eta)\left[h^{mi}(\my)(\sigma^{T})_{i}^{j}(\my,\eta)-
 h^{mi}(\mz)(\sigma^{T})_{i}^{j}(\mz,\zeta)\right]\times \\ 
&\bar{\chi}_{v_n}(\mz,\zeta)\tilde{\chi}_{u_n}(\my,\eta) 
\pa_{k}\rho_{\eps,\delta}(\mx-\my,\xi-\eta)\pa_{j}\rho_{\eps,\delta}(\mx-\mz,\xi-\zeta)d\my d\eta d\mz d\zeta d\mu(\mx)d\xi \\
&+
\int_{\mathbb{R}^{3d+3}}  \delta_{ml}h^{rl}(\mz)(\sigma^{T})_{r}^{k}(\mz,\zeta)\left[h^{mi}(\mz)(\sigma^{T})_{i}^{j}(\mz,\zeta)-
h^{mi}(\my)(\sigma^{T})_{i}^{j}(\my,\eta)\right]\times \\ 
&\bar{\chi}_{v_n}(\mz,\zeta)\tilde{\chi}_{u_n}(\my,\eta) \pa_{j}\rho_{\eps,\delta}(\mx-\mz,\xi-\zeta) 
\pa_{k}\rho_{\eps,\delta}(\mx-\my,\xi-\eta)d\my d\eta d\mz d\zeta d\mu(\mx)d\xi \\
&
\qquad\qquad +\int_{\mathbb{R}^{d+1}} \bar{\chi}^{\eps,\delta}_{v_n}\tilde{\chi}^{\eps,\delta}_{u_n} G(\mx,\xi) d\mu d\xi =\\
&\int_{\mathbb{R}^{3d+3}} \delta_{lm} \left[ h^{rl}(\my)(\sigma^{T})_{r}^{k}(\my,\eta)-h^{rl}(\mz)(\sigma^{T})_{r}^{k}(\mz,\zeta) \right] 
\times \\ 
&\qquad\qquad\qquad \times \left[ h^{mi}(\my)(\sigma^{T})_{i}^{j}(\my,\eta)-h^{mi}(\mz)(\sigma^{T})_{i}^{j}(\mz,\zeta) \right] 
\bar{\chi}_{v_n}(\mz,\zeta)\tilde{\chi}_{u_n}(\my,\eta)\\
&
\qquad\qquad\qquad\times \pa_{k}\rho_{\eps,\delta}(\mx-\my,\xi-\eta)\pa_{j}
\rho_{\eps,\delta}(\mx-\mz,\xi-\zeta)d\my d\eta d\mz d\zeta d\mu(\mx)d\xi\\ 
&
\qquad\qquad\qquad\qquad\qquad +\int_{\mathbb{R}^{d+1}} \bar{\chi}^{\eps,\delta}_{v_n}\tilde{\chi}^{\eps,\delta}_{u_n} G(\mx,\xi) d\mu d\xi,
\label{eq:num2.6}
\end{split} 
\end{equation}
again for some bounded function $G$.

Expanding the functions $h^{rl}(\sigma^T)^k_r$ in Taylor series  (see the appendix for the 
details of the following calculation), we conclude that 
\begin{equation}
\begin{split}
&\int_{\mathbb{R}^{3d+3}} \delta_{lm} \left[ h^{rl}(\my)(\sigma^{T})_{r}^{k}(\my,\eta)-h^{rl}(\mz)(\sigma^{T})_{r}^{k}(\mz,\zeta) \right] \times \\ 
&\qquad\qquad\left[ h^{mi}(\my)(\sigma^{T})_{i}^{j}(\my,\eta)-h^{mi}(\mz)(\sigma^{T})_{i}^{j}(\mz,\zeta) \right] 
\bar{\chi}_{v_n}(\mz,\zeta)\tilde{\chi}_{u_n}(\my,\eta)\times\\
&\qquad\qquad\quad\times \pa_{k}\rho_{\eps,\delta}(\mx-\my,\xi-\eta)\pa_{j}\rho_{\eps,\delta}(\mx-\mz,\xi-\zeta)d\my d\eta d\mz d\zeta d\mu(\mx)d\xi\\ 
&\approx 
-\int_{\mathbb{R}^{d+1}} \delta_{lm} \bar{\chi}^{\eps,\delta}_{v_n} \tilde{\chi}^{\eps,\delta}_{u_n} 
\left[ \pa_{k}\left(h^{rl}(\sigma^{T})_{r}^{k}\right)\pa_{j}\left(h^{im}(\sigma^{T})_{i}^{j}\right) \right.   \\
& \qquad\qquad\qquad\qquad\qquad\qquad\qquad+ \left.\pa_{j}\left(h^{rl}(\sigma^{T})_{r}^{k}\right)\pa_{k}\left(h^{mi}(\sigma^{T})_{i}^{j}\right) \right] d\mu d\xi
\\ 
& =
\int_{\mathbb{R}^{d+1}} \bar{\chi}^{\eps,\delta}_{v_n}\tilde{\chi}^{\eps,\delta}_{u_n} G(\mx,\xi) d\mu d\xi,
\label{eq:num2.7}
\end{split}
\end{equation} 
where $\approx$ holds if we let first $n\to \infty$, then $\delta\to0$ and finally $\eps\to 0$ (so that all other terms in the Taylor expansion will go to zero). 
 This shows that the third and fourth term of (\ref{eq:num2}) again simply sum 
 to a term of the form $\int \bar \chi^{\eps,\delta}_{v_n} \tilde\chi^{\eps,\delta}_{u_n} G(\mx,\xi) d\mu d\xi$.

Next, an integration by parts shows that the fifth and sixth term in (\ref{eq:num2}) sum to
\begin{multline*}
-\int_{\mathbb{R}^{d+1}}\bar{\chi}_{v_n}^{\eps,\delta}\tilde{\chi}_{u_n}^{\eps,\delta}\pa_{j}\left[g^{ij}\Gamma_{ml}^{m}a_{i}^{l}+\Gamma_{ka}^{j}g^{ia}a_{i}^{k}\right]d\mu d\xi=\int_{\mathbb{R}^{d+1}}\bar{\chi}_{v_n}^{\eps,\delta}\tilde{\chi}_{u_n}^{\eps,\delta}G(\mx,\xi)d\mu d\xi,
\end{multline*}
hence it only remains to study the first two terms in \eqref{eq:num2}.

The sum of the first and second term from \eqref{eq:num2} can be shown to be (approximately) equal to
\begin{equation}
\begin{split}
 &\int_{\mathbb{R}^{d+1}}(1-\chi_{v_n}^{\eps,\delta})\chi_{u_n}^{\eps,\delta}G(\mx,\xi)d\mu d\xi 
+2 \int_{\mathbb{R}^{d+1}} \delta_{ml} \left(\phi_\alpha h^{rl}\pa_{k}((\sigma^{T})_{r}^{k}\chi_{u_n})\right)^{\eps,\delta} \times\\
&
\qquad\qquad\qquad\qquad \times\left(\phi_\alpha h^{mi} \pa_j((\sigma^{T})_{i}^{j}(1-\chi_{v_n}))\right)^{\eps,\delta}d\mu d\xi .\label{eq:num4}
\end{split}
\end{equation} 
Again, this uses the product rule and integration by parts and the details are in the appendix. 

Note that, similarly,
\begin{equation}
\begin{split}
\int_{\mathbb{R}^{d+1}} \delta_{ml} \left(h^{rl}\phi_\alpha \chi_{u_n}\pa_{k}(\sigma^{T})_{r}^{k}\right)^{\eps,\delta}\left(h^{mi} \phi_\alpha \pa_j((\sigma^{T})_{i}^{j}(1-\chi_{v_n}))\right)^{\eps,\delta} d\mu d\xi \\+
\int_{\mathbb{R}^{d+1}} \delta_{ml} \left(\phi_\alpha h^{rl}\pa_{k}((\sigma^{T})_{r}^{k}\chi_{u_n})\right)^{\eps,\delta}\left(h^{mi} \phi_\alpha (1-\chi_{v_n}) \pa_{j}(\sigma^{T})_{i}^{j}\right)^{\eps,\delta} d\mu d\xi \\
  \approx \int_{\mathbb{R}^{d+1}}(1-\chi_{v_n}^{\eps,\delta})\chi_{u_n}^{\eps,\delta}G(\mx,\xi)d\mu d\xi 
\label{eq:num5}
\end{split}
\end{equation}
and obviously 
\begin{equation}
\begin{split}
\int_{\mathbb{R}^{d+1}} \delta_{ml} \left(h^{rl}\phi_\alpha \chi_{u_n}\pa_{k}(\sigma^{T})_{r}^{k}\right)^{\eps,\delta} \left(h^{mi} \phi_\alpha (1-\chi_{v_n}) \pa_{j}(\sigma^{T})_{i}^{j}\right)^{\eps,\delta}d\mu d\xi  \\
  \approx \int_{\mathbb{R}^{d+1}}(1-\chi_{v_n}^{\eps,\delta})\chi_{u_n}^{\eps,\delta}G(\mx,\xi)d\mu d\xi.
\label{eq:num5.5}
\end{split}
\end{equation}
Using \eqref{eq:num5} and \eqref{eq:num5.5} we then conclude that (\ref{eq:num4}) can be written as
\begin{equation}
\begin{split}
& 2\int_{\mathbb{R}^{d+1}} \delta_{ml} \left(\phi_\alpha h^{rl}\pa_{k}((\sigma^{T})_{r}^{k}\chi_{u_n})\right)^{\eps,\delta} 
\left(h^{mi}\phi_\alpha \pa_j((\sigma^{T})_{i}^{j}(1-\chi_{v_n}))\right)^{\eps,\delta}d\mu d\xi  \\ 
& \qquad \approx
2\int_{\mathbb{R}^{d+1}} \delta_{ml} \left(\phi_\alpha h^{rl}\pa_{k}((\sigma^{T})_{r}^{k}\chi_{u_n})-
\phi_\alpha h^{rl}\chi_{u_n}\pa_{k}(\sigma^{T})_{r}^{k}\right)^{\eps,\delta} \\ 
&\qquad\qquad\times
\left(h^{mi} \phi_\alpha \pa_j((\sigma^{T})_{i}^{j}(1-\chi_{v_n}))-h^{mi}\phi_\alpha (1-\chi_{v_n})
\pa_j(\sigma^{T})_{i}^{j}\right)^{\eps,\delta}d\mu d\xi 
\\
&\qquad\qquad\qquad\qquad\qquad+\int_{\mathbb{R}^{d+1}}(1-\chi_{v_n}^{\eps,\delta})\chi_{u_n}^{\eps,\delta}G(\mx,\xi)d\mu d\xi  \\
&\qquad=-2\int_{\mathbb{R}^{d+1}} \delta_{ml}\left(\phi_\alpha h^{rl}\pa_{k}((\sigma^{T})_{r}^{k}\chi_{u_n})
-\phi_\alpha h^{rl}\chi_{u_n}\pa_{k}(\sigma^{T})_{r}^{k}\right)^{\eps,\delta} \\ 
&\qquad\qquad\qquad\times
\left(h^{mi}\phi_\alpha \pa_j((\sigma^{T})_{i}^{j}\chi_{v_n})-h^{mi}\phi_\alpha \chi_{v_n}
\pa_j(\sigma^{T})_{i}^{j}\right)^{\eps,\delta}d\mu d\xi 
\\
&\qquad\qquad\qquad\qquad\qquad+\int_{\mathbb{R}^{d+1}}(1-\chi_{v_n}^{\eps,\delta})\chi_{u_n}^{\eps,\delta}G(\mx,\xi)d\mu d\xi .
\label{eq:num6}
\end{split}
\end{equation}

Now in \eqref{beta} we defined $\beta^{\psi}(\mx,\xi)$ by $\left(\pa_{\xi}\beta^{\psi}\right)(\mx,\xi)=\sqrt{\psi(\xi)}\sigma^{T}(\mx,\xi)$ 
and $\beta^{\psi}(\mx,0)$ $=0$
for any $\mx$, and \eqref{integralofchi} gives
\begin{equation*}
\begin{split}
&\left(h^{mi}\phi_\alpha \pa_j((\sigma^{T})^j_i\chi_{v_n})\right)\star\rho_{\eps,\delta}(\mx,\xi)=\int_{\mathbb{R}^{d}}  
(h^{mi}(\my) \phi_\alpha(\my) \pa_{j}\rho_{\eps}(\mx-\my) \\
&\qquad\qquad\qquad\quad-\pa_j(h^{mi}\phi_\alpha)(\my)\rho_{\eps}(\mx-\my)) 
\int_{\mathbb{R}} \rho_{\delta}(\xi-\eta)(\sigma^{T})_{i}^{j}(\my,\eta)\chi_{v_n}(\my,\eta)d\eta d\my\\ 
&=
\int_{\mathbb{R}^{d}}  (h^{mi}(\my)\phi_\alpha(\my) \pa_{j}\rho_{\eps}(\mx-\my)-\pa_j(h^{mi}\phi_\alpha)(\my)\rho_{\eps}(\mx-\my)) 
(\beta^{\rho^2_{\delta}(\xi-.)})_{i}^{j}(\my,v_n(\my))d\my \\ 
&=
\int_{\mathbb{R}^{d}}  h^{mi}(\my)\phi_\alpha(\my) \rho_{\eps}(\mx-\my)\pa_{j}\left((\beta^{\rho^2_{\delta}(\xi-.)})_{i}^{j}(\my,v_n(\my))\right) d\my,
\end{split}
\end{equation*}
and
\begin{equation*}
\begin{split}
&\left(h^{mi}\phi_\alpha \chi_{v_n} \pa_j(\sigma^{T})^j_i\right)\star\rho_{\eps,\delta}(\mx,\xi)
=\int_{\mathbb{R}^{d}} h^{mi}(\my)\phi_\alpha(\my)  \rho_{\eps}(\mx-\my)\times \\
&\qquad\qquad\qquad\qquad\qquad\qquad\qquad\quad
\times\int_{\mathbb{R}} \rho_{\delta}(\xi-\eta)\pa_j(\sigma^{T})_{i}^{j}(\my,\eta) \chi_{v_n}(\my,\eta))d\eta d\my\\ 
&\qquad\qquad\qquad\qquad=
\int_{\mathbb{R}^{d}}  h^{mi}(\my)\phi_\alpha(\my) \rho_{\eps}(\mx-\my)\pa_{j}\left((\beta^{\rho^2_{\delta}(\xi-.)})_{i}^{j}(\my,\zeta)\right)|_{\zeta=v_n(\my)}d\my.
\end{split}
\end{equation*}
 Hence using \eqref{8} it follows that their difference is given by 
\begin{equation*}
\begin{split}
&\left(h^{mi}\phi_\alpha \pa_j((\sigma^{T})^j_i\chi_{v_n})-h^{mi}\phi_\alpha \chi_{v_n} 
\pa_j(\sigma^{T})^j_i\right)\star\rho_{\eps,\delta}(\mx,\xi)\\ 
&\qquad\qquad=
\int_{\mathbb{R}^{d}} h^{mi}(\my)\phi_\alpha(\my) \rho_{\eps}(\mx-\my) \left(\Div_{\my}\left(\beta^{\rho^2_{\delta}(\xi-.)}(\my,v_n(\my))\right) \right.\\
&\qquad\qquad\qquad\qquad\qquad\qquad\qquad-\left.\left(\Div_{\my}\left(\beta^{\rho^2_{\delta}(\xi-.)}(\my,\zeta)\right)\right)|_{\zeta=v_n(\my)}\right)_{i}d\my.
\end{split}
\end{equation*}
An analogous treatment of $\left(\phi_\alpha h^{rl}\pa_{k}((\sigma^{T})_{r}^{k}\chi_{u_n})-
\phi_\alpha h^{rl}\chi_{u_n}\pa_{k}(\sigma^{T})_{r}^{k}\right)^{\eps,\delta}$ shows that (\ref{eq:num6}) becomes
\begin{equation*}
\begin{split}
&-2\int_{\mathbb{R}^{d+1}} \delta_{ml} \left(\phi_\alpha h^{rl}\pa_{k}((\sigma^{T})_{r}^{k}\chi_{u_n})-
\phi_\alpha h^{rl}\chi_{u_n}\pa_{k}(\sigma^{T})_{r}^{k}\right)^{\eps,\delta} \\ 
&\qquad\qquad\qquad\qquad\qquad\times
\left(h^{mi}\phi_\alpha \pa_j((\sigma^{T})_{i}^{j}\chi_{v_n})-h^{mi}\phi_\alpha \chi_{v_n}\pa_j(\sigma^{T})_{i}^{j}\right)^{\eps,\delta}d\mu d\xi \\ 
&=
-2\int_{\mathbb{R}^{3d+1}} \delta_{ml}h^{mi}(\my)h^{rl}(\my')\phi_\alpha(\my)\phi_\alpha(\my')\rho_{\eps}(\mx-\my)\rho_{\eps}(\mx-\my')\times \\ 
& \qquad\times \left(\Div_{\my'} \left(\beta^{\rho_{\delta}^2(\xi-.)}(\my',v_n(\my'))\right)-
\left(\Div_{\my'}\left(\beta^{\rho_{\delta}^2(\xi-.)}(\my',\zeta)\right)\right)|_{\zeta=v_n(\my')}\right)_{r}\times\\
&\times\left(\Div_{\my}\left(\beta^{\rho_{\delta}^2(\xi-.)}(\my,u_n(\my))\right)-
\left(\Div_{\my}\left(\beta^{\rho_{\delta}^2(\xi-.)}(\my,\zeta)\right)\right)|_{\zeta=u_n(\my)}\right)_{i}d\my  d\my' d\mu d\xi.
\end{split}
\end{equation*}
This finally establishes \eqref{est-L1}.
\end{proof}
 Before we state the next lemma, we note that by a limiting procedure (exactly as in \cite[(2.7)]{CP}) we may
insert $S(u)={\rm sgn}_+(\xi)(u-\xi)_++{\rm sgn}_+(-\xi)(u-\xi)_-$  
into \eqref{entr-ineq-prep}. Then multiplying by a test function in $\xi$ and integrating over 
$(t,\mx)\in [0,\infty)\times M$ as well as over $\xi$ it follows that
\begin{equation}
\label{eq:boundedness of measures}
\int_{\R^+\times M}(n_{u_n}+m_{u_n})(t,\mx,\xi)dtd\mu(\mx)\leq \nu(\xi),
\end{equation}
in the sense of distributions in $\xi$, where
\begin{align}
\nu(\xi) &:={\rm sgn}_+(\xi) \|(u_0-\xi)_+\|_{L^1(M)}+{\rm sgn}_+(-\xi)\|(u_0-\xi)_-\|_{L^1(M)} \\ &=
{\rm sgn}_+(\xi) \|(u_0-\xi)_+\|_{L^1(M)}
\end{align}
(which is a bounded function compactly supported in $[0,1]$). Since this holds for all $n$, it must also hold for the weak limit $n_u+m_u$. 

\begin{lemma} 
\label{lemma:estimate for ueta}
For weakly convergent subsequences $\chi_{u_n}, n_{u_n}, m_{u_n}$ and $\chi_{v_n},n_{v_n},m_{v_n}$ (as in Lemma \ref{subseqconverge}) we have 
\begin{equation} \label{eq:estimate for ueta}
\begin{split}
&\int_{\R^{d}}\int_0^\infty (\phi_\alpha (m_{u_n}+n_{u_n}))^{\eps,\delta} \pa_{\xi} (\phi_\alpha \chi_{v_n})^{\eps,\delta}  \\
&\qquad\qquad\qquad\qquad\qquad\qquad+(\phi_\alpha (m_{v_n} + n_{v_n}))^{\eps,\delta} \pa_{\xi} (\phi_\alpha \chi_{u_n})^{\eps,\delta}d\mu(\mx) d\xi \\
& +2\int_{\mathbb{R}^{3d+2}}\int_0^\infty \delta_{lm} h^{mi}(\my)h^{rl}(\my')\phi_\alpha(\my)\phi_\alpha(\my') \rho_{\eps}(t-\tau,\mx-y)
\rho_{\eps}(t-\tau',\mx-\my')\times  \\ 
&\times \left(\Div_{\my'} \left(\beta^{\rho^2_{\delta}(\xi-.)}(\my',v_n(\tau',\my'))\right)-\left(\Div_{\my'}\left(\beta^{\rho^2_{\delta}(\xi-.)}(\my',\zeta)\right)\right)|_{\zeta=v_n(\tau',\my')}\right)_{r} \\
&\times\left(\Div_{\my}\left(\beta^{\rho^2_{\delta}(\xi-.)}(\my,u_n(\tau,\my))\right)\right.\\
&\qquad-\left.\left(\Div_{\my}\left(\beta^{\rho^2_{\delta}(\xi-.)}(\my,\zeta)\right)\right)|_{\zeta=u_n(\tau,\my)}\right)_{i} 
 d\my d\tau d\my' d\tau' d\mu d\xi 
\leq 0.
\end{split}
\end{equation} 
\end{lemma}
\begin{proof} 
To begin with, a straightforward calculation using \eqref{integralofchi} shows that
\begin{equation}
\label{derivative of chi}
\pa_\xi(\phi_\alpha \chi_u)^{\eps,\delta}=\phi_\alpha^\eps(\mx)\rho_\delta(\xi)-(\phi_\alpha \delta(\xi-u))\star \rho_{\eps,\delta}.
\end{equation}
Therefore,  $(\phi_\alpha (m_{u_n}+n_{u_n}))^{\eps,\delta} \pa_{\xi} (\phi_\alpha \chi_{v_n})^{\eps,\delta}$ splits into the following terms:
\begin{equation}\label{eq:estimate n and m}
\begin{split}
&\int_{\R^{d}}\int_0^\infty (\phi_\alpha (m_{u_n}+n_{u_n}))^{\eps,\delta} \pa_{\xi} (\phi_\alpha \chi_{v_n})^{\eps,\delta} d\mu(\mx) d\xi =\\
&\qquad\qquad\qquad\int_{\R^{d}}\int_0^\infty \phi_\alpha^\eps(\mx) \rho_\delta(\xi) \left(\phi_\alpha(n_{u_n}+m_{u_n}) \right)^{\eps,\delta}d\mu d\xi \\
&\qquad\qquad\qquad\qquad\qquad-\int_{\R^{d}}\int_0^\infty (\phi_\alpha n_{u_n})^{\eps,\delta} (\phi_\alpha \delta(\xi-u_n))^{\eps,\delta} d\mu d\xi  \\
&\qquad\qquad\qquad\qquad\qquad\qquad-\int_{\R^{d}}\int_0^\infty (\phi_\alpha m_{u_n})^{\eps,\delta} (\phi_\alpha \delta(\xi-u_n))^{\eps,\delta} d\mu d\xi  \\
&\le \int_{\R^{d}}\int_0^\infty \phi_\alpha^\eps(\mx) \rho_\delta(\xi) \left(\phi_\alpha(n_{u_n}+m_{u_n}) \right)^{\eps,\delta}d\mu d\xi\\
&\qquad\qquad\qquad\qquad\qquad\qquad -\int_{\R^{d}}\int_0^\infty (\phi_\alpha n_{u_n})^{\eps,\delta} (\phi_\alpha \delta(\xi-u_n))^{\eps,\delta} d\mu d\xi
\end{split}
\end{equation}
by positivity of $m_{u_n}$.  First note that the first term is zero since $\rho_\delta $ is supported in $(-1,0)$.

We now look at the second term. 
By definition of $n_{u_n}$ (see \eqref{eq:def n})
\begin{align*}
&\int_{\R^{d}}\int_0^\infty (\phi_\alpha n_{u_n})^{\eps,\delta} (\phi_\alpha \delta(\xi-u_n))^{\eps,\delta} d\mu d\xi \\
&= \int_{\R^{3d+2}}\int_0^\infty \rho_\delta(\xi-u_n(\tau,\my))\rho_\delta(\xi-u_n(\tau',\my'))\rho_\eps(t-\tau,\mx-\my)\rho_\eps(t-\tau',\mx-\my')\times \\
& \times \phi_\alpha(\my)\phi_\alpha(\my')
\left|\Div(\beta(\my,u_n(\tau,\my)))-\Div(\beta(\my,\zeta))\Big|_{\zeta=u_n(\tau,\my)} \right|^2_{g(\my)} d\my d\tau d\my' d\tau' d\mu d\xi.
\nonumber
\end{align*} 
Defining the vector field $X_{u_n}$ on $\R^d$ by
\begin{equation*}
X^i_{u_n}(t,\my):=h^{ij}(\my)\left( \Div(\beta(\my,u_n(t,\my))-\Div(\beta(\my,\zeta))\Big|_{\zeta=u_n(t,\my)} \right)_j
\end{equation*} we see that
\begin{equation*}
\left|\Div(\beta(\my,u_n(t,\my))-\Div(\beta(\my,\zeta))\Big|_{\zeta=u_n(t,\my)} \right|^2_{g(\my)}=\left|X_{u_n}(t,\my)\right|^2_e,
\end{equation*}
where $|.|_e$ denotes the Euclidean norm on $\R^d$. So, using 
$|X_{u_n}(t,\my)|_e^2+|X_{v_n}(t',\my')|_g^2\geq 2 \delta_{ij} X_{u_n}^i(t,\my) X_{u_n}^j(t,\my')$ and the chain rule \eqref{crule} 
(which holds since $u_n,v_n$ are sufficiently regular for all $n$) we see that
\begin{align} \label{useofchainrule}
&-\int_{\R^{d}}\int_0^\infty (\phi_\alpha n_{u_n})^{\eps,\delta} (\phi_\alpha \delta(\xi-v_n))^{\eps,\delta}+ (\phi_\alpha n_{v_n})^{\eps,\delta} 
(\phi_\alpha \delta(\xi-u_n))^{\eps,\delta}d\mu d\xi \leq \\
& -2\int_{\mathbb{R}^{3d+2}}\int_0^\infty \delta_{lm} h^{mi}(\my)h^{rl}(\my')\phi_\alpha(\my)\phi_\alpha(\my') \rho_{\eps}(t-\tau,\mx-\my)
\rho_{\eps}(t-\tau',\mx-\my')\times \nonumber \\ 
&\times \left(\Div_{\my'} \left(\beta^{\rho^2_{\delta}(\xi-.)}(\my',v_n(\tau',\my'))\right)-
\left(\Div_{\my'}\left(\beta^{\rho^2_{\delta}(\xi-.)}(\my',\zeta)\right)\right)|_{\zeta=v_n(\tau',\my')}\right)_{r}\times \nonumber\\
&\times\left(\Div_{\my}\left(\beta^{\rho^2_{\delta}(\xi-.)}(\my,u_n(\tau,\my))\right) \right. \nonumber \\
& \qquad\qquad\qquad\qquad- \left.\left(\Div_{\my}
\left(\beta^{\rho^2_{\delta}(\xi-.)}(\my,\zeta)\right)\right)|_{\zeta=u_n(\tau,\my)}\right)_{i}d(\my,\tau,\my',\tau',\mu,\xi). \nonumber
\end{align} 
This concludes the proof.
\end{proof}

 From this we conclude that condition \eqref{secondorder} is fulfilled. 

\begin{lemma} 
\label{lemma:secondorder}
Under the assumptions of the previous Lemma the limits satisfy the estimate \eqref{secondorder}.
\end{lemma}
\begin{proof} As before, due to the presence of the cut-off functions $\phi_\alpha$, 
we may without loss of generality suppose that $M=\R^d$.  
We first calculate
\begin{multline} \label{new2}
\int_{\R^{d}}\int_0^\infty (\phi_\alpha (m_{u}+n_{u}))^{\eps,\delta} \pa_{\xi} (\phi_\alpha \chi_{v})^{\eps,\delta} d\mu(\mx) d\xi =\\
-\int_{\R^{d}}\int_0^\infty (\phi_\alpha (m_{u}+n_{u}))^{\eps,\delta} \pa_{\xi} 
(\phi_\alpha (1-\chi_{v}))^{\eps,\delta} d\mu(\mx) d\xi\\
=
-\int_{\R^{d}}(\phi_\alpha (m_{u}+n_{u}))^{\eps,\delta}(t,\mx,0) 
(\phi_\alpha (1- \chi_{v}))^{\eps,\delta} d\mu(\mx) d\xi \hfill  \\ 
+
\int_{\R^{d}}\int_0^\infty \pa_\xi(\phi_\alpha (m_{u}+n_{u}))^{\eps,\delta} (\phi_\alpha (1-\chi_{v}))^{\eps,\delta} d\mu(\mx) d\xi. 
\end{multline} 
Next, note that $(\phi_\alpha (n_u+m_u))^\eps$ is continuous (and even locally Lipschitz) in $\xi$ since by assumption $(\chi_u,u_0,m_u,n_u)$ satisfies \eqref{kin-eq}, hence \eqref{kinetic-loc}, which implies that $\pa_\xi(\phi_\alpha (n_u+m_u))^\eps$ will be in $L^\infty_\text{loc}([0,\infty)\times \R^d \times \R)$. Thus
\begin{equation*}
\int_{\R^{d}}\int_0^\infty \phi_\alpha^\eps(\mx) \rho_\delta(\xi) \left(\phi_\alpha(n_{u}+m_{u}) \right)^{\eps,\delta}d\mu d\xi \to \int_{\R^{d}} \phi_\alpha^\eps(\mx) \left(\phi_\alpha(n_{u}+m_{u}) \right)^{\eps}(\mx,0)d\mu 
\end{equation*}
as $\delta \to 0$. Now, for any the estimate \eqref{eq:boundedness of measures} (and $\nu(\xi)=0$ for $\xi <0$) shows that the measure $\int_{\mathbb{R}^+\times M} (n_u+m_u)(t,\mx,\xi) dt d\mu$ is supported in $[0,\infty)$ hence by positivity $(n_u+m_u)(t,\mx,\xi)$ is supported in $[0,\infty)\times M \times [0,\infty)$. 
But this implies $(n_u+m_u)^{\eps}(t,\mx,\xi)=0$ on $[0,\infty) \times M$ for any $\xi <0$. 
Thus 
\begin{equation}\label{numu0}
(n_u+m_u)^{\eps}(t,\mx,0)=0
\end{equation} 
on $[0,\infty) \times M$ by continuity, so 
\begin{equation} \label{eq:mnatzero}\int_0^t \int_{\R^{d}} \phi_\alpha^\eps(\mx) \left(\phi_\alpha(n_{u}+m_{u}) \right)^{\eps}(\tau,\mx,0)d\mu d\tau=0.
\end{equation} 
Since $0\leq 1-\chi_v \leq 1$ (and $m_u, n_u$ and $\phi_\alpha $ are non-negative) this immediately implies that the first term in \eqref{new2} must converge to zero as $\delta \to 0$ as well.

Thus,
\begin{align*}
 &\limsup_{\eps \to 0}\limsup_{\delta \to 0} 
\int_{\R^{d}}\int_0^\infty(\phi_\alpha (m_{u}+n_{u}))^{\eps,\delta} \pa_{\xi} (\phi_\alpha \chi_{v})^{\eps,\delta} d\mu(\mx) d\xi \\
& \qquad\qquad= \limsup_{\eps \to 0}\limsup_{\delta \to 0}
\int_{\R^{d}}\int_0^\infty\pa_\xi(\phi_\alpha (m_{u}+n_{u}))^{\eps,\delta} (\phi_\alpha (1-\chi_{v}))^{\eps,\delta} d\mu(\mx) d\xi.
\end{align*}
Combining this with Lemma \ref{lemma1} and Lemma \ref{lemma:estimate for ueta} and letting $n \to \infty$ (keeping in mind \eqref{eq:boundedness of measures}) gives the claim. \end{proof}

From the above, we see that the following theorem holds.

\begin{theorem}
\label{main-thm} Denote by ${\cal F}$ the set of all tuples $(\chi_u, u_0,m_u, n_u)$ obtained as the weak limits along subsequences as in Lemma \ref{subseqconverge} of appropriate terms from the vanishing viscosity approximation \eqref{VV} with initial condition \eqref{ic}.
Then ${\cal F}$ satisfies the conditions from Definition \ref{def-kinetic}.
\end{theorem}
\begin{proof}
That such limits satisfy \eqref{kin-eq}, \eqref{kin-ic} is part of the statement of Lemma \ref{subseqconverge}, while relation \eqref{secondorder} follows from Lemma \ref{lemma:secondorder}.
\end{proof}
 As a direct consequence we obtain the following result on the uniqueness of entropy solutions:
\begin{corollary} Let $u$, $v$ be entropy solutions of \eqref{main-eq}, \eqref{ic}. Then $u=v$.
\end{corollary}
\begin{proof} We do this by showing that the set $\mathcal{F}$ consisting of all entropy solutions is kinetically admissible. From this, uniqueness of entropy solutions follows from Theorem \ref{uniq-k}.
	
	As was shown in Section \ref{diffprelsec}, the kinetic functions $\chi_u$, $\chi_v$ corresponding
to $u$, $v$ satisfy the Cauchy problem \eqref{kin-eq}, \eqref{kin-ic}. It remains to show \eqref{secondorder}. But this follows as in Lemmas  \ref{lemma1} to \ref{lemma:secondorder} by replacing the sequences there with the constant sequences $\chi_u,n_u,m_u$ and $\chi_v,n_v,m_v$: Note that the only place where the higher regularity of the $u_{n}$ enters is in the use of the chain rule in \eqref{useofchainrule}, which entropy solutions have to satisfy by definition.
\end{proof}

The final theorem of the paper establishes existence of the entropy admissible solutions to \eqref{main-eq}, \eqref{ic}.

\begin{theorem}
There exists a function $u: [0,\infty)\times M \to [0,1]$ satisfying the conditions of Definition \ref{def-entropy}. 
It is obtained as the strong $L^1_{loc}([0,\infty)\times M)$ limit of the functions $(u_\eta)$ obtained as the solution to \eqref{VV}, \eqref{ic}.
\end{theorem}
\begin{proof}  From Theorem \ref{main-thm}, Lemma \ref{subseqconverge} and Theorem \ref{uniq-k} it follows that for the entire family $(u_\eta)$ (and not only a subsequence)
$$
{\rm sgn}_+(u_\eta(t,\mx)-\xi) \rightharpoonup  {\rm sgn}_+(u(t,\mx)-\xi) \ \ {\rm as} \ \ \eta\to 0 \ \ {\rm in} \ \ L^\infty([0,\infty)\times M \times [0,\infty)),
$$ 
where $u$ is defined in \eqref{***}. Indeed, according to Theorem \ref{main-thm} and Lemma \ref{subseqconverge}, any  weak-$\star$ limit of a subsequence of 
${\rm sgn}_+(u_\eta(t,\mx)-\xi)$ belongs to the family ${\cal F}$ from Definition \ref{def-kinetic}, while from Theorem \ref{uniq-k} it follows that all such limits coincide (since they correspond to the same initial value). 
This in turn means that the Young measure corresponding to the family $(u_\eta)$ is the atomic measure of the form $\delta(u(t,\mx)-\xi)$. Indeed, for any $f\in C^1(\R)$, we have (keeping in mind \eqref{***} and the fact that $ 0\leq u_\eta \leq 1$):
\begin{align*}
&f(u_\eta(t,\mx))=\int_0^{u_\eta(t,\mx)} f'(\xi)d\xi+f(0) =\int_0^1 f'(\xi) {\rm sgn}_+(u_\eta(t,\mx)-\xi) d\xi+f(0) \\
& \underset{\scriptscriptstyle\eta \to 0}{{-\!\!\!\rightharpoonup}}  
\int_0^1 f'(\xi){\rm sgn}_+(u(t,\mx)-\xi)d\xi +f(0) =\int_{\R} f(\xi) \delta(u(t,\mx)-\xi) d\xi .
\end{align*} 
From here, according to  standard properties  of Young measures \cite{Dpe}, we conclude that 
$$
u_\eta \to u \ \ \text{strongly in} \ \ L^1_{loc}([0,\infty)\times M).
$$ 
The strong convergence provides all the conditions from Definition \ref{def-entropy} (cf.\ \cite[Section 7]{CP}).
\end{proof}

{\bf Acknowledgment:} This work is supported by OEAD-project ME 07/2015-16, and FWF-project P28770 of the Austrian
Science Fund. It is also supported by the Croatian Science Foundation's project "Weak convergence methods and applications"
number 9780, and by Austrian-Montenegro bi-lateral project "Flow problems on manifolds". Melanie Graf is the recipient of a DOC Fellowship of the Austrian Academy of Sciences at the Institute of Mathematics, University of Vienna. We are indebted to the referee for several comments that have led to valuable improvements.

\renewcommand{\theequation}{A-\arabic{equation}}
  \setcounter{equation}{0}   
  \section*{Appendix}  
In this appendix we provide details of the arguments used in the proof of Lemma
\ref{lemma1}. As before, we shall suppress the $t$-dependence to simplify the presentation. Also, as in the Lemma, $\int d\xi$ is to
 be understood to in fact mean $\int_{\R^+} d\xi$.

To begin with, we show \eqref{eq:num0}, using integration by parts, \eqref{divom}, and Lemma \ref{friedrichs}, as well
as the fact that $\pa_k\sqrt{|g|}=\Gamma^s_{ks}\sqrt{|g|}$, and
\begin{equation}\label{g}
\pa_jg^{ij}= -g^{ia}\Gamma^j_{aj}-g^{jb}\Gamma^i_{jb}.
\end{equation} 
\begin{multline}
\int_{\mathbb{R}^{d+1}}\Div\Div(\tilde{\chi}_{u_\eta}a)\star\rho_{\eps,\delta}\bar{\chi}_{v_n}^{\eps,\delta}d\mu d\xi=\\
\int_{\mathbb{R}^{d+1}}(g^{ij}\pa_j\Div(\tilde{\chi}_{u_n}a)_{i})\star\rho_{\eps,\delta}\bar{\chi}_{v_n}^{\eps,\delta}d\mu d\xi-\int_{\mathbb{R}^{d+1}}\left(g^{ij}\Gamma_{ij}^{k}\Div(\tilde{\chi}_{u_n}a)_{k}\right)\star\rho_{\eps,\delta}\bar{\chi}_{v_n}^{\eps,\delta}d\mu d\xi\\
=\int_{\mathbb{R}^{2d+2}}g^{ij}(\my)\pa_{j}\Div(\tilde{\chi}_{u_n}a)_{i}(\my,\eta)\rho_{\eps,\delta}(\mx-\my,\xi-\eta)d\my d\eta \bar{\chi}_{v_n}^{\eps,\delta}(\mx,\xi)\sqrt{\left|g(\mx)\right|}d\mx d\xi d\my d\eta\\
-\int_{\mathbb{R}^{d+1}}\left(g^{ij}\Gamma_{ij}^{k}\Div(\tilde{\chi}_{u_n}a)_{k}\right)\star\rho_{\eps,\delta}\bar{\chi}_{v_n}^{\eps,\delta}d\mu d\xi\\
 =\int_{\mathbb{R}^{2d+2}}g^{ij}(\my)\Div(\tilde{\chi}_{u_n}a)_{i}(\my,\eta)\pa_{j}\rho_{\eps,\delta}(\mx-\my,\xi-\eta)d\my d\eta \bar{\chi}_{v_n}^{\eps,\delta}(\mx,\xi)\sqrt{\left|g(\mx)\right|}d\mx d\xi d\my d\eta \\
-\int_{\mathbb{R}^{2d+2}}\pa_{j}g^{ij}(\my)\Div(\tilde{\chi}_{u_n}a)_{i}(\my,\eta)\rho_{\eps,\delta}(\mx-\my,\xi-\eta)dyd\eta \bar{\chi}_{v_n}^{\eps,\delta}(x,\xi)\sqrt{\left|g(\mx)\right|}d\mx d\xi d\my d\eta\\
-\int_{\mathbb{R}^{d+1}}\left(g^{ij}\Gamma_{ij}^{k}\Div(\chi_{u}a)_{k}\right)\star\rho_{\eps,\delta}\bar{\chi}_{v_n}^{\eps,\delta}d\mu d\xi=\\
=-\int_{\mathbb{R}^{2d+2}}g^{ij}(y)\Div(\tilde{\chi}_{u_n}a)_{i}(\my,\eta)\rho_{\eps,\delta}(\mx-\my,\xi-\eta)dyd\eta\pa_{j}\bar{\chi}_{v_n}^{\eps,\delta}(\mx,\xi)\sqrt{\left|g(\mx)\right|}d\mx d\xi d\my d\eta \\
\hspace{1.3em} -\int_{\mathbb{R}^{2d+2}}g^{ij}(\my)\Div(\tilde{\chi}_{u_n}a)_{i}(\my,\eta)\rho_{\eps,\delta}(\mx-\my,\xi-\eta)d\my d\eta\bar{\chi}_{v_n}^{\eps,\delta}(\mx,\xi)\pa_{j}\sqrt{\left|g(\mx)\right|}d\mx d\xi d\my d\eta\\
+\int_{\mathbb{R}^{d+1}}\left((g^{ia}\Gamma_{aj}^{j}+g^{jb}\Gamma_{jb}^{i})\Div(\tilde{\chi}_{u_n}a)_{i}\right)\star\rho_{\eps,\delta}\bar{\chi}_{v_n}^{\eps,\delta}d\mu d\xi\\
-\int_{\mathbb{R}^{d+1}}\left(g^{ij}\Gamma_{ij}^{k}\Div(\tilde{\chi}_{u_n}a)_{k}\right)\star\rho_{\eps,\delta}\bar{\chi}_{v_n}^{\eps,\delta}d\mu d\xi\\
=\int_{\mathbb{R}^{d+1}}\left(g^{ij}\Div(\tilde{\chi}_{u_n}a)_{i}\right)\star\rho_{\eps,\delta}\pa_{j}\bar{\chi}_{v_n}^{\eps,\delta}d\mu d\xi-\int_{\mathbb{R}^{d+1}}\left(g^{ij}\Div(\tilde{\chi}_{u_n}a)_{i}\right)\star\rho_{\eps,\delta}\bar{\chi}_{v_n}^{\eps,\delta}\Gamma_{js}^{s}d\mu d\xi+\\
+\int_{\mathbb{R}^{d+1}}\left((g^{ia}\Gamma_{aj}^{j}+g^{jb}\Gamma_{jb}^{i})\Div(\tilde{\chi}_{u_n}a)_{i}\right)\star\rho_{\eps,\delta}\bar{\chi}_{v}^{\eps,\delta}d\mu d\xi\\
-\int_{\mathbb{R}^{d+1}}\left(g^{ij}\Gamma_{ij}^{k}\Div(\chi_{u}a)_{k}\right)\star\rho_{\eps,\delta}\bar{\chi}_{v}^{\eps,\delta}d\mu d\xi\\
\approx \int_{\mathbb{R}^{d+1}}\left(g^{ij}\Div(\tilde{\chi}_{u_n}a)_{i}\right)\star\rho_{\eps,\delta}\pa_{j}\bar{\chi}_{v_n}^{\eps,\delta}d\mu d\xi
\label{divdivstart}
\end{multline}

Turning now to \eqref{eq:num1}, we again use \eqref{g}, as well as
\begin{equation}\label{k}
g^{ij}a^k_i=g^{ri}(\sigma^{T})_{r}^{k} (\sigma^{T})_{i}^{j}
\end{equation}
to calculate:
\begin{multline}
\!\!\!\!\!\!
\int_{\mathbb{R}^{d+1}}\left(g^{ij}\Div(\tilde{\chi}_{u_n}a)_{i}\right)\star\rho_{\eps,\delta}\pa_{j}\bar{\chi}_{v_n}^{\eps,\delta}d\mu d\xi\\
=\int_{\mathbb{R}^{2d+2}}g^{ij}(\my)\pa_{k}(\tilde{\chi}_{u_n}a_{i}^{k})(\my,\eta)\rho_{\eps,\delta}(\mx-\my,\xi-\eta)d\my d\eta\pa_{j}\bar{\chi}_{v_n}^{\eps,\delta}(\mx,\xi)d\mu(\mx)d\xi d\my d\eta \\
+\int_{\mathbb{R}^{2d+2}}\pa_{j}\bar{\chi}_{v_n}^{\eps,\delta}(\mx,\xi)\tilde{\chi}_{u_n}(\my,\eta)g^{ij}(\my)\left[\Gamma_{ml}^{m}(\my)a_{i}^{l}(\my,\eta)-\Gamma_{li}^{m}(\my)a_{m}^{l}(\my,\eta)\right]\times \\
\hspace*{\fill} \times \rho_{\eps,\delta}(\mx-\my,\xi-\eta) d\my d\eta d\mu(\mx)d\xi\\
=
\int_{\mathbb{R}^{2d+2}}g^{ij}(\my)(\chi_{u}a_{i}^{k})(\my,\eta)\pa_{k}\rho_{\eps,\delta}(\mx-\my,\xi-\eta)d\my d\eta\pa_{j}\bar{\chi}_{v_n}^{\eps,\delta}(\mx,\xi) d\my d\eta d\mu(\mx)d\xi \\
-\int_{\mathbb{R}^{2d+2}}\pa_{k}g^{ij}(\my)(\tilde{\chi}_{u_n}a_{i}^{k})(\my,\eta)\rho_{\eps,\delta}(\mx-\my,\xi-\eta)dyd\eta\pa_{j}\bar{\chi}_{v_n}^{\eps,\delta}(\mx,\xi) d\my d\eta d\mu(\mx)d\xi\\
+\int_{\mathbb{R}^{2d+2}}\pa_{j}\bar{\chi}_{v_n}^{\eps,\delta}(\mx,\xi)\tilde{\chi}_{u_n}(\my,\eta)g^{ij}(\my)\left[\Gamma_{ml}^{m}(\my)a_{i}^{l}(\my,\eta)-\Gamma_{li}^{m}(\my)a_{m}^{l}(\my,\eta)\right]\times \\
\hspace*{\fill}
\times \rho_{\eps,\delta}(\mx-\my,\xi-\eta)d\my d\eta d\mu(\mx)d\xi\\=
\int_{\mathbb{R}^{2d+2}}g^{ij}(\my)(\chi_{u}a_{i}^{k})(\my,\eta)\pa_{k}\rho_{\eps,\delta}(\mx-\my,\xi-\eta)dyd\eta\pa_{j}\chi_{v}^{\eps,\delta}(\mx,\xi) d\my d\eta d\mu(\mx)d\xi\\
-\int_{\mathbb{R}^{2d+2}}\left(\Gamma_{ka}^{j}g^{ia}+\Gamma_{kb}^{i}g^{jb}\right)
(\my)(\tilde{\chi}_{u_n} a_{i}^{k})(\my,\eta)\times \\
\hspace*{\fill}\times
\rho_{\eps,\delta}(\mx-\my,\xi-\eta)d\my d\eta\pa_{j}\bar{\chi}_{v_n}^{\eps,\delta}(\mx,\xi) d\my d\eta d\mu(\mx)d\xi\\
+\int_{\mathbb{R}^{2d+2}}\pa_{j}\bar{\chi}_{v_n}^{\eps,\delta}(\mx,\xi)\tilde{\chi}_{u_n}(\my,\eta)g^{ij}(\my)
\left[\Gamma_{ml}^{m}(\my)a_{i}^{l}(\my,\eta)-\Gamma_{li}^{m}(\my)a_{m}^{l}(\my,\eta)\right]	
\times\\
\hspace*{\fill}\times
\rho_{\eps,\delta}(\mx-\my,\xi-\eta)d\my d\eta d\mu(\mx)d\xi\\
=
\int_{\mathbb{R}^{d+1}}\int_{\mathbb{R}^{d+1}}g^{ri}(\my)(\sigma^{T})_{r}^{k}(\my,\eta)(\sigma^{T})_{i}^{j}(\my,\eta)
\bar{\chi}_{v_n}(\mz,\zeta)\tilde{\chi}_{u_n}(\my,\eta)\times\\
\hspace*{\fill}\times
\pa_{k}\rho_{\eps,\delta}(\mx-\my,\xi-\eta)\pa_{j}\rho_{\eps,\delta}(\mx-\mz,\xi-\zeta)d\my d\eta d\mz d\zeta d\mu(\mx)d\xi\\
+\int_{\mathbb{R}^{d+1}}\int_{\mathbb{R}^{d+1}}\pa_{j}\bar{\chi}_{v_n}^{\eps,\delta}(\mx,\xi)\tilde{\chi}_{u_n}(\my,\eta)\left[g^{ij}\Gamma_{ml}^{m}a_{i}^{l}+\Gamma_{ka}^{j}g^{ia}a_{i}^{k}\right](\my,\eta)\times\\
\hspace*{\fill}\times
\rho_{\eps,\delta}(\mx-\my,\xi-\eta)d\my d\eta d\mu(\mx)d\xi\\ 
\approx
\int_{\mathbb{R}^{d+1}}\int_{\mathbb{R}^{d+1}}g^{ri}(\my)(\sigma^{T})_{r}^{k}(\my,\eta)(\sigma^{T})_{i}^{j}(\my,\eta)\bar{\chi}_{v_n}(\mz,\zeta)
\tilde{\chi}_{u_n}(\my,\eta)\times \\
\hspace*{\fill}\times
\pa_{k}\rho_{\eps,\delta}(\mx-\my,\xi-\eta)\pa_{j}\rho_{\eps,\delta}(\mx-\mz,\xi-\zeta)d\my d\eta d\mz d\zeta d\mu(\mx)d\xi\\
+\int_{\mathbb{R}^{d+1}}\pa_{j}\chi_{v}^{\eps,\delta}(\mx,\xi)\chi_{u}^{\eps,\delta}(\mx,\xi)\left[g^{ij}\Gamma_{ml}^{m}a_{i}^{l}+\Gamma_{ka}^{j}g^{ia}a_{i}^{k}\right](\mx,\xi)d\mu(\mx)d\xi. \label{eq:no1}
\end{multline}
Here, the last $\approx$ follows from the the Friedrichs lemma in the following way: For any $f\in C^2_c(\R^{d+1})$ we have
\begin{align*}
\int_{\mathbb{R}^{d+1}}\pa_{j}\bar{\chi}_{v_n}^{\eps,\delta}(x,\xi) & \left(\left(\tilde{\chi}_{u_n}f\right)\star\rho_{\eps,\delta}\right)(\mx,\xi)d\mu(\mx)d\xi\\ 
&=-\int_{\mathbb{R}^{d+1}}\bar{\chi}_{v_n}^{\eps,\delta}(\mx,\xi)\left(\pa_{j}(\tilde{\chi}_{u_n}f)\star\rho_{\eps,\delta}\right)(\mx,\xi)d\mu(\mx)d\xi\\
& \hspace*{3.1em}
-\int_{\mathbb{R}^{d+1}}\bar{\chi}_{v_n}^{\eps,\delta}(\mx,\xi)\left((\tilde{\chi}_{u_n}f)\star\rho_{\eps,\delta}\right)(\mx,\xi)\pa_{j}
\sqrt{\left|g(\mx)\right|}d\mx d\xi\\ 
&\approx
-\int_{\mathbb{R}^{d+1}}\bar{\chi}_{v_n}^{\eps,\delta}(\mx,\xi)\pa_{j}(\tilde{\chi}_{u_n}^{\eps,\delta}f)(\mx,\xi)d\mu(\mx)d\xi\\
& \hspace*{3.1em}
-\int_{\mathbb{R}^{d+1}}\bar{\chi}_{v_n}^{\eps,\delta}(\mx,\xi)\tilde{\chi}_{u_n}^{\eps,\delta}(\mx,\xi)f(\mx,\xi)\pa_{j}
\sqrt{\left|g(\mx)\right|}d\mx d\xi \\ 
&= \int_{\mathbb{R}^{d+1}}\pa_{j}\bar{\chi}_{v_n}^{\eps,\delta}(\mx,\xi)\tilde{\chi}_{u_n}^{\eps,\delta}(\mx,\xi)\, f(\mx,\xi)d\mu(\mx)d\xi.
\end{align*}
To summarize, (\ref{eq:no1}) becomes
\begin{multline}
\int_{\mathbb{R}^{d+1}}\left(g^{ij}\Div(\tilde{\chi}_{u_n}a)_{i}\right)\star\rho_{\eps,\delta}\pa_{j}\bar{\chi}_{v_n}^{\eps,\delta} d\mu d\xi\\ 
\hspace*{-1em}\approx
\int_{\mathbb{R}^{3d+3}}g^{ri}(\my)(\sigma^{T})_{r}^{k}(\my,\eta)(\sigma^{T})_{i}^{j}(\my,\eta)\bar{\chi}_{v_n}(\mz,\zeta)
\tilde{\chi}_{u_n}(\my,\eta)\times\\
\hspace*{3em}\times
\pa_{k}\rho_{\eps,\delta}(\mx-\my,\xi-\eta)
\pa_{j}\rho_{\eps,\delta}(\mx-\mz,\xi-\zeta)d\my d\eta d\mz d\zeta d\mu(\mx)d\xi\\
\hspace*{\fill}
+\int_{\mathbb{R}^{d+1}}\pa_{j}\bar{\chi}_{v_n}^{\eps,\delta}(\mx,\xi)\tilde{\chi}_{u_n}^{\eps,\delta}(\mx,\xi)
\left[g^{ij}\Gamma_{ml}^{m}a_{i}^{l}+\Gamma_{ka}^{j}g^{ia}a_{i}^{k}\right](\mx,\xi)d\mu(\mx)d\xi. \label{a3new}
\end{multline}

To simplify notation we set $\tilde{\sigma}^{ij}:=h^{ik} (\sigma^T)^j_k$. Looking at the fourth term from (\ref{eq:num2}) we see
\begin{multline}
\int_{\R^{3d+3}} \delta_{ml} \tilde{\sigma}^{lk}(\mz,\zeta)\left[\tilde{\sigma}^{mj}(\mz,\zeta)-\tilde{\sigma}^{mj}(\my,\eta)\right] \times \hspace*{\fill} \\
\bar{\chi}_{v_n}(\mz,\zeta)\tilde{\chi}_{u_n}(\my,\eta)\pa_{k}\rho_{\eps,\delta}(
\mx-\mz,\xi-\zeta)\pa_{j}\rho_{\eps,\delta}(\mx-\my,\xi-\eta)d\my d\eta d\mz d\zeta d\mu(\mx)d\xi\\ =
\int_{\R^{3d+3}} \delta_{ml} \tilde{\sigma}^{lk}(\mz,\zeta)\tilde{\sigma}^{mj}(\mz,\zeta)\bar{\chi}_{v_n}(\mz,\zeta)\tilde{\chi}_{u_n}(\my,\eta)\times \hspace*{\fill} \\
\hspace*{\fill} \pa_{k}\rho_{\eps,\delta}(
\mx-\mz,\xi-\zeta)\pa_{j}\rho_{\eps,\delta}(\mx-\my,\xi-\eta)d\my d\eta d\mz d\zeta d\mu(\mx)d\xi\\ 
\hspace*{-10em}
-\int_{\R^{3d+3}} \delta_{ml} \tilde{\sigma}^{lk}(\mz,\zeta) \tilde{\sigma}^{mj}(\my,\eta)\bar{\chi}_{v_n}(\mz,\zeta)\tilde{\chi}_{u_n}(\my,\eta) \times \\
\hspace*{\fill} \pa_{k}\rho_{\eps,\delta}(
\mx-\mz,\xi-\zeta)\pa_{j}\rho_{\eps,\delta}(\mx-\my,\xi-\eta)d\my d\eta d\mz d\zeta d\mu(\mx)d\xi\\ 
\hspace*{-15.5em}
= \int_{\R^{3d+3}} \delta_{ml} \tilde{\sigma}^{lk}(\mz,\zeta)\tilde{\sigma}^{mj}(\mz,\zeta)\bar{\chi}_{v_n}(\mz,\zeta)\tilde{\chi}_{u_n}(\my,\eta)\times \\
\hspace*{\fill} \pa_{k}\rho_{\eps,\delta}(
\mx-\mz,\xi-\zeta)\pa_{j}\rho_{\eps,\delta}(\mx-\my,\xi-\eta)d\my d\eta d\mz d\zeta d\mu(\mx)d\xi\\ 
\hspace*{-7em}
-\int_{\R^{d+1}} \delta_{ml} \pa_k (\tilde{\sigma}^{lk} \bar{\chi}_{v_n})\star\rho_{\eps,\delta} \pa_j (\tilde{\sigma}^{mj}\tilde{\chi}_{u_n})\star \rho_{\eps,\delta} \sqrt{|g|} d\mx d\xi\\ 
\hspace*{-15em}
= \int_{\R^{3d+3}} \delta_{ml} \tilde{\sigma}^{lk}(\mz,\zeta)\tilde{\sigma}^{mj}(\mz,\zeta)\bar{\chi}_{v_n}(\mz,\zeta)\tilde{\chi}_{u_n}(\my,\eta)\times \\
\hspace*{\fill} \pa_{j}\rho_{\eps,\delta}(
\mx-\mz,\xi-\zeta)\pa_{k}\rho_{\eps,\delta}(\mx-\my,\xi-\eta)d\my d\eta d\mz d\zeta d\mu(\mx)d\xi\\ 
\hspace*{-7em}
+\int_{\R^{d+1}} \delta_{ml} (\tilde{\sigma}^{lk} \bar{\chi}_{v_n})\star\rho_{\eps,\delta} \pa_k\pa_j (\tilde{\sigma}^{mj}\tilde{\chi}_{u_n})\star 
\rho_{\eps,\delta} \sqrt{|g|} d\mx  d\xi \\
\hspace*{\fill}
+\int_{\R^{d+1}} \delta_{ml} (\tilde{\sigma}^{lk} \bar{\chi}_{v_n})\star\rho_{\eps,\delta} \pa_j (\tilde{\sigma}^{mj}\tilde{\chi}_{u_n})\star \rho_{\eps,\delta} \Gamma^s_{ks}d\mx d\xi\\ 
\hspace*{-14.5	em}
\approx
\int_{\R^{3d+3}} \delta_{ml} \tilde{\sigma}^{lk}(\mz,\zeta)\tilde{\sigma}^{mj}(\mz,\zeta)\bar{\chi}_{v_n}(\mz,\zeta)\tilde{\chi}_{u_n}(\my,\eta)\times \\
\hspace*{\fill} \pa_{j}\rho_{\eps,\delta}(
\mx-\mz,\xi-\zeta)\pa_{k}\rho_{\eps,\delta}(\mx-\my,\xi-\eta)d\my d\eta d\mz d\zeta d\mu(\mx)d\xi\\ 
\hspace*{\fill}
-\int_{\R^{d+1}} \delta_{ml} \pa_j(\tilde{\sigma}^{lk} \bar{\chi}_{v_n})\star\rho_{\eps,\delta} \pa_k 
(\tilde{\sigma}^{mj}\tilde{\chi}_{u_n})\star \rho_{\eps,\delta}d\mu(\mx)d\xi\\
-\int_{\R^{d+1}} \delta_{ml} (\tilde{\sigma}^{lk} \bar{\chi}_{v_n})\star\rho_{\eps,\delta} 
\pa_k (\tilde{\sigma}^{mj}\tilde{\chi}_{u_n})\star \rho_{\eps,\delta} \Gamma^s_{js}d\mx d\xi
\\ 
\hspace*{\fill}+
\int_{\R^{d+1}} \delta_{ml} \tilde{\sigma}^{lk} \bar{\chi}_{v_n}^{\eps,\delta} \pa_j (\tilde{\sigma}^{mj}\tilde{\chi}_{u_n}^{\eps,\delta}) \Gamma^s_{ks}d\mx d\xi \\ 
\hspace*{-15.5em}\approx
\int_{\R^{3d+3}} \delta_{ml} \tilde{\sigma}^{lk}(\mz,\zeta)\left[\tilde{\sigma}^{mj}(\mz,\zeta)-\tilde{\sigma}^{mj}(\my,\eta)\right] \times \\
\hspace*{\fill} \bar{\chi}_{v_n}(\mz,\zeta)\tilde{\chi}_{u_n}(\my,\eta)\pa_{j}\rho_{\eps,\delta}(
\mx-\mz,\xi-\zeta)\pa_{k}\rho_{\eps,\delta}(\mx-\my,\xi-\eta)d\my d\eta d\mz d\zeta d\mu(\mx)d\xi\\ 
\hspace*{\fill}
-
\int_{\R^{d+1}} \delta_{ml} \tilde{\sigma}^{lk} \bar{\chi}_{v_n}^{\eps,\delta} \pa_k (\tilde{\sigma}^{mj}\tilde{\chi}_{u_n}^{\eps,\delta}) \Gamma^s_{js}d\mx d\xi
+
\int_{\R^{d+1}} \delta_{ml} \tilde{\sigma}^{lk} \bar{\chi}_{v_n}^{\eps,\delta} \pa_j (\tilde{\sigma}^{mj}\tilde{\chi}_{u_n}^{\eps,\delta}) \Gamma^s_{ks}d\mx d\xi  \\ 
\hspace*{-2.5em} =
\int_{\mathbb{R}^{3d+3}}  \delta_{ml}h^{rl}(\mz)(\sigma^{T})_{r}^{k}(\mz,\zeta)\left[h^{mi}(\mz)(\sigma^{T})_{i}^{j}(\mz,\zeta)-h^{mi}(\my)(\sigma^{T})_{i}^{j}(\my,\eta)\right]\times \\ \hspace*{\fill} \bar{\chi}_{v_n}(\mz,\zeta)\tilde{\chi}_{u_n}(\my,\eta) \pa_{j}\rho_{\eps,\delta}(\mx-\mz,\xi-\zeta) \pa_{k}\rho_{\eps,\delta}(\mx-\my,\xi-\eta)d\my d\eta d\mz d\zeta d\mu(\mx)d\xi \\ +
\int_{\mathbb{R}^{d+1}} \delta_{ml}h^{rl}\bar{\chi}^{\eps,\delta}_{v_n}\tilde{\chi}^{\eps,\delta}_{u_n}(\sigma^{T})_{r}^{k}\Gamma_{ks}^{s}\pa_{j}(h^{mi}(\sigma^{T})_{i}^{j})d\mu d\xi
\\-\int_{\mathbb{R}^{d+1}} \delta_{ml}h^{rl}\bar{\chi}^{\eps,\delta}_{v_n}\tilde{\chi}^{\eps,\delta}_{u_n}(\sigma^{T})_{r}^{k}\Gamma_{js}^{s}\pa_{k}(h^{mi}(\sigma^{T})_{i}^{j})d\mu d\xi.
\label{a4new}
\end{multline}
This establishes \eqref{eq:num2.5}.

Next we show \eqref{eq:num2.7}. Using again the notation $\tilde{\sigma}^{ij}:=h^{ik} (\sigma^T)^j_k$, we have to show that
\begin{multline}
\int_{\mathbb{R}^{3d+3}} \delta_{lm} \left[ \tilde{\sigma}^{lk}(\my,\eta)-\tilde{\sigma}^{lk}(\mz,\zeta) \right]\, \left[ \tilde{\sigma}^{mj}(\my,\eta)-\tilde{\sigma}^{mj}(\mz,\zeta) \right] \times\\
\times\bar{\chi}_{v_n}(\mz,\zeta)\tilde{\chi}_{u_n}(\my,\eta)\pa_{k}\rho_{\eps,\delta}(\mx-\my,\xi-\eta)\pa_{j}\rho_{\eps,\delta}(\mx-\mz,\xi-\zeta)d\my d\eta d\mz d\zeta d\mu(\mx)d\xi\\ \approx
-\int_{\mathbb{R}^{d+1}} \delta_{lm} \bar{\chi}^{\eps,\delta}_{v_n} \tilde{\chi}^{\eps,\delta}_{u_n} \left[ \pa_{k}\left(\tilde{\sigma}^{lk}\right)\pa_{j}\left(\tilde{\sigma}^{mj}\right)+\pa_{j}\left(\tilde{\sigma}^{lk}\right)\pa_{k}\left(\tilde{\sigma}^{mj}\right) \right] d\mu d\xi.
\label{a5new}
\end{multline} 
To do so, we introduce a change of variables, 
$$
\bar{\my}=\frac{\mx-\my}{\eps}, \ \ \bar{\mz}=\frac{\mx-\mz}{\eps}, \ \ \bar{\eta}=\frac{\xi-\eta}{\delta}, \ \ \bar{\zeta}=\frac{\xi-\zeta}{\delta},
$$ 
so the  left  hand side of \eqref{a5new} becomes
\begin{equation}
\begin{split}
&(-1)^{2d+2} \int_K \varepsilon^{2d}\delta^{2}\,\tilde{\chi}_{u_n}(\mx-\varepsilon\bar{\my},\xi-\delta\bar{\eta})
\bar{\chi}_{v_n}(\mx-\varepsilon\bar{\mz},\xi-\delta\bar{\zeta})\times\\
& \hspace*{3em}
\delta_{lm} \left[ \tilde{\sigma}^{lk}(\mx-\varepsilon\bar{\my},\xi-\delta\bar{\eta})-
\tilde{\sigma}^{lk}(\mx-\varepsilon\bar{\mz},\xi-\delta\bar{\zeta}) \right] \times \\
& \hspace*{6em}
\times \left[ \tilde{\sigma}^{mj}(\mx-\varepsilon\bar{\my},\xi-\delta\bar{\eta})-\tilde{\sigma}^{mj}(\mx-\varepsilon\bar{\mz},\xi-\delta\bar{\zeta}) \right] \\
&\hspace*{8em}
\times \pa_{k}\rho_{\eps,\delta}(\eps \bar{\my},\delta \bar{\eta})\pa_{j}\rho_{\eps,\delta}(\eps \bar{\mz},\delta \bar{\zeta}) \: d\bar{\mz}d\bar{\zeta}d\bar{\my}d\bar{\eta}d\mu(\mx)d\xi,
\label{a5.5new}
\end{split}
\end{equation}
 where $K\subset \mathbb{R}^{3(d+1)}$ is a suitable compact set (the $\tilde{\chi}_{u_n}$ 
have compact support, uniformly in $n$). Henceforth, we will simply use the letter 
$K$ to generically denote such compact sets.
Recalling our simplifying assumption on suppressing $t$-dependence, we have  
$$
\partial_{j}\rho_{\varepsilon,\delta}(\varepsilon\bar{\my},\delta\bar{\eta})
=\frac{1}{\varepsilon^{d}\delta}\omega_1(\bar{\eta})\,\Pi_{s\neq j}\omega_2(\bar{\my}_{s})\,\frac{1}{\varepsilon}\,
\partial_{j}\omega_2(\bar{\my}_{j}),
$$  
so \eqref{a5.5new} becomes
\begin{equation}
\begin{split}
&\int_K \frac{1}{\eps^2}\,\tilde{\chi}_{u_n}(\mx-\varepsilon\bar{\my},\xi-\delta\bar{\eta})\bar{\chi}_{v_n}
(\mx-\varepsilon\bar{\mz},\xi-\delta\bar{\zeta})\times\\
& \hspace*{2em}\times\delta_{lm} \left[ \tilde{\sigma}^{lk}(\mx-\varepsilon\bar{\my},\xi-\delta\bar{\eta})
-\tilde{\sigma}^{lk}(\mx-\varepsilon\bar{\mz},\xi-\delta\bar{\zeta}) \right]\, \\
& \hspace*{4em}\times \left[ \tilde{\sigma}^{mj}(\mx-\varepsilon\bar{\my},\xi-\delta\bar{\eta})-
\tilde{\sigma}^{mj}(\mx-\varepsilon\bar{\mz},\xi-\delta\bar{\zeta}) \right] \omega_1(\bar{\eta})\,\Pi_{s\neq k}\omega_2(\bar{\my}_{s}) \\
& \hspace*{7em} \times
\partial_{k}\omega_2(\bar{\my}_{k}) 
\omega_1(\bar{\zeta})\,\Pi_{r\neq j}\omega_2(\bar{\mz}_{r})\,\partial_{j}\omega_2(\bar{\mz}_{j}) \: 
d\bar{\mz}d\bar{\zeta}d\bar{\my}d\bar{\eta}d\mu(\mx)d\xi.
\label{a5.6new}
\end{split}
\end{equation}
We now expand $\tilde{\sigma}^{lk}(\mx-\eps \tilde{\my},\xi-\delta \tilde{\eta})$ and $\tilde{\sigma}^{lk}(\mx-\eps \tilde{\mz},\xi-\delta \tilde{\zeta})$ in a Taylor series around $(\mx,\xi)$ to obtain 
\begin{multline*}
\tilde{\sigma}^{lk}(\mx-\eps \bar{\my},\xi-\delta \bar{\eta})-\tilde{\sigma}^{lk}(\mx-\eps \bar{\mz},\xi-\delta \bar{\zeta})=\\
\sum_{r=1}^{d}\partial_{r}\tilde{\sigma}^{lk}(\mx,\xi)\,\varepsilon(\bar{\mz}_{r}-\bar{\my}_{r})+\partial_{\xi}\tilde{\sigma}^{lk}(\mx,\xi)\,\delta(\bar{\zeta}-\bar{\eta})+\\
+\sum_{|\alpha|=2}\left[ R^{lk}_\alpha(\mx,\eps \bar{\my},\xi,\delta \bar{\eta})(-\eps \bar{\my},-\delta \bar{\eta})^\alpha- 
R^{lk}_\alpha(\mx,\eps \bar{\mz},\xi,\delta \bar{\zeta})(-\eps \bar{\mz},-\delta \bar{\zeta})^\alpha\right],
\end{multline*} 
where the $R^{lk}_\alpha$ are suitable bounded functions (since $\tilde{\sigma}^{lk}\in C^2$).
Doing the same for $\tilde{\sigma}^{mj}(\mx-\eps \tilde{\my},\xi-\delta \tilde{\eta})$ and $\tilde{\sigma}^{mj}(\mx-\eps \tilde{\mz},\xi-\delta \tilde{\zeta})$ and multiplying, we see that the only relevant remaining term is 
$$\eps^2 \left[\sum_{r=1}^{d}\partial_{r}\tilde{\sigma}^{lk}(\mx,\xi)\,(\bar{\mz}_{r}-\bar{\my}_{r})\right]\times \left[\sum_{s=1}^{d}\partial_{s}\tilde{\sigma}^{mj}(\mx,\xi)\,(\bar{\mz}_{s}-\bar{\my}_{s})\right]$$
since all other terms will go to zero as,  first, $n \to \infty$,  then $\delta \to 0$ and finally $\eps \to 0$ by boundedness on compact sets (uniformly in $n$) of all functions  appearing in the integrand.
We may also replace $\tilde{\chi}_{u_n}(\mx-\varepsilon\bar{\my},\xi-\delta
\bar{\eta})$ by $\tilde{\chi}_{u_n}(\mx,\xi)$: The difference of both versions can 
be estimated by
$$
C\, \int_{{K}} \int_{B(0,\eps)}\int_{B(0,\delta)}\left|\tilde{\chi}_{u}
(\mx-\varepsilon\bar{\my},\xi-\delta\bar{\eta})-\tilde{\chi}_{u}(\mx,\xi)\right|d\bar{\eta} d\bar{\my} d\mu(\mx)d\xi
$$ 
as  $n \to \infty$.  Now by assumption $\tilde{\chi}_u \in L^1$ since it is bounded and has compact support, so the Lebesgue differentiation theorem applies, and together with dominated convergence shows that this integral converges to zero as, first, $\delta \to 0$, and then $\eps \to 0$. 
By a similar argument we may afterwards also replace $\bar{\chi}_{v_n}(\mx-\varepsilon\bar{\mz},\xi-\delta
\bar{\zeta})$ by $\bar{\chi}_{v_n}(\mx,\xi)$.
This gives
\begin{multline}
\int_K \frac{1}{\eps^2}\,\tilde{\chi}_{u_n}(\mx-\varepsilon\bar{\my},\xi-\delta\bar{\eta})\bar{\chi}_{v_n}(\mx-\varepsilon\bar{\mz},\xi-\delta\bar{\zeta})\times\\
\hspace*{-8em}\delta_{lm} \left[ \tilde{\sigma}^{lk}(\mx-\varepsilon\bar{\my},\xi-\delta\bar{\eta})-\tilde{\sigma}^{lk}(\mx-\varepsilon\bar{\mz},\xi-\delta\bar{\zeta}) \right]
\times\\
\times \left[ \tilde{\sigma}^{mj}(\mx-\varepsilon\bar{\my},\xi-\delta\bar{\eta})-\tilde{\sigma}^{mj}(\mx-\varepsilon\bar{\mz},
\xi-\delta\bar{\zeta}) \right] \omega_1(\bar{\eta})\,\Pi_{s\neq k}\omega_2(\bar{\my}_{s})\\
\times
\partial_{k}\omega_2(\bar{\my}_{k}) \omega_1(\bar{\zeta})\,\Pi_{r\neq j}\omega_2(\bar{\mz}_{r})\,\partial_{j}\omega_2(\bar{\mz}_{j}) \: d\bar{\mz}d\bar{\zeta}d\bar{\my}d\bar{\eta}d\mu(\mx)d\xi \\ \hspace*{-8em}\approx
\int_K \tilde{\chi}_{u_n}(\mx,\xi)\bar{\chi}_{v_n}(\mx,\xi) 
\delta_{lm} \sum_{r,s=1}^{d}\left[\partial_{r}\tilde{\sigma}^{lk}(\mx,\xi)
 \partial_{s}\tilde{\sigma}^{mj}(\mx,\xi)\right. \times\\ 
\left. \times (\bar{\mz}_{s}\bar{\mz}_r-\bar{\my}_{s}\bar{\mz}_r-\bar{\mz}_{s}\bar{\my}_r+
\bar{\my}_{s}\bar{\my}_r)\right] 
\omega_1(\bar{\eta})\,\Pi_{s\neq k}\omega_2(\bar{\my}_{s})\,\partial_{k}\omega_2(\bar{\my}_{k}) \omega_1(\bar{\zeta})\\
\times \Pi_{r\neq j}\omega_2(\bar{\mz}_{r})
\partial_{j}\omega_2(\bar{\mz}_{j}) \: d\bar{\mz}d\bar{\zeta}d\bar{\my}d\bar{\eta}d\mu(\mx)d\xi \\
\hspace*{-8em}
=\int_K \tilde{\chi}_{u_n}(\mx,\xi)\bar{\chi}_{v_n}(\mx,\xi) 
\delta_{lm} \left[-\partial_{j}\tilde{\sigma}^{lk}(\mx,\xi)\,\partial_{k}\tilde{\sigma}^{mj}(\mx,\xi)\,\bar{\my}_{k}\bar{\mz}_j \right.
\\ \left.
-\partial_{k}\tilde{\sigma}^{lk}(\mx,\xi)\,\partial_{j}\tilde{\sigma}^{mj}(\mx,\xi)\,\bar{\mz}_{j}\bar{\my}_k)\right] 
\partial_{k}\omega_2(\bar{\my}_{k})\partial_{j}\omega_2(\bar{\mz}_{j}) \: d\bar{\mz}_j d\bar{\my}_k d\mu(\mx)d\xi \\
\hspace*{-10em}=
\int_K \tilde{\chi}_{u_n}(\mx,\xi)\bar{\chi}_{v_n}(\mx,\xi) 
\delta_{lm} \left[-\partial_{j}\tilde{\sigma}^{lk}(\mx,\xi)\,\partial_{k}\tilde{\sigma}^{mj}(\mx,\xi)\right.\\
\left.-\partial_{k}\tilde{\sigma}^{lk}(\mx,\xi)\,\partial_{j}\tilde{\sigma}^{mj}(\mx,\xi))\right] \: d\mu(\mx)d\xi.
\label{a6new}
\end{multline}
This concludes the proof of \eqref{eq:num2.7}.

Next we have to show that the first and second term of \eqref{eq:num2} sum to \eqref{eq:num4}. For the first term of \eqref{eq:num2} we get
\begin{multline}
\int_{\mathbb{R}^{3d+3}} \delta_{ml}h^{rl}(\my)(\sigma^{T})_{r}^{k}(\my,\eta)h^{mi}(\mz)(\sigma^{T})_{i}^{j}(\mz,\zeta)  \bar{\chi}_{v_n}(\mz,\zeta)\tilde{\chi}_{u_n}(\my,\eta)  \times\\
\times \pa_{k}\rho_{\eps,\delta}(\mx-\my,\xi-\eta)\pa_{j}\rho_{\eps,\delta}(\mx-\mz,\xi-\zeta)d\my d\eta d\mz d\zeta d\mu(\mx)d\xi \\=
\int_{\mathbb{R}^{d+1}} \delta_{ml} \pa_{k}\left(h^{rl}(\sigma^{T})_{r}^{k}\tilde{\chi}_{u_n}\right)\star\rho_{\eps,\delta}\pa_{j}\left(h^{mi}(\sigma^{T})_{i}^{j}\bar{\chi}_{v_n}\right)\star\rho_{\eps,\delta}d\mu d\xi \\ 
\approx
\int_{\mathbb{R}^{d+1}} \delta_{ml} \pa_{k}(\phi_\alpha h^{rl})(\sigma^{T})_{r}^{k}\chi_{u_n}^{\eps,\delta}\pa_{j}\left(h^{mi}(\sigma^{T})_{i}^{j}\bar{\chi}_{v_n}^{\eps,\delta}\right)d\mu d\xi \\+
\int_{\mathbb{R}^{d+1}} \delta_{ml}\left(\phi_\alpha h^{rl}\pa_{k}((\sigma^{T})_{r}^{k}\chi_{u_n})\right)^{\eps,\delta}\pa_{j}\left(h^{mi}(\sigma^{T})_{i}^{j}\bar{\chi}_{v_n}\right)^{\eps,\delta}d\mu d\xi \\ 
\hspace*{-5em}\approx
\int_{\mathbb{R}^{d+1}} \delta_{ml} \pa_{k}(\phi_\alpha h^{rl})(\sigma^{T})_{r}^{k}\chi_{u_n}^{\eps,\delta}h^{mi}(\sigma^{T})_{i}^{j}\phi_\alpha \pa_{j}(1-\chi_{v_n}^{\eps,\delta})d\mu d\xi \\ 
\hspace*{-13em}
+\int_{\mathbb{R}^{d+1}}(1-\chi_{v_n}^{\eps,\delta})\chi_{u_n}^{\eps,\delta}G(\mx,\xi)d\mu d\xi \\
+\int_{\mathbb{R}^{d+1}} \delta_{ml} \left(\phi_\alpha h^{rl}\pa_{k}((\sigma^{T})_{r}^{k}\chi_{u_n})\right)^{\eps,\delta}
\pa_{j}\left(h^{mi}(\sigma^{T})_{i}^{j}\bar{\chi}_{v_n}\right)^{\eps,\delta}d\mu d\xi 
\label{eq:num3}
\end{multline}
A similar calculation gives
\begin{multline}
\int_{\mathbb{R}^{d+1}} \delta_{ml}\left(\phi_\alpha h^{rl}\pa_{k}((\sigma^{T})_{r}^{k}\chi_{u_n})\right)^{\eps,\delta}\pa_{j}\left(h^{mi}(\sigma^{T})_{i}^{j}\bar{\chi}_{v_n}\right)^{\eps,\delta}d\mu d\xi \\ \approx
\int_{\mathbb{R}^{d+1}} \delta_{ml} \phi_\alpha h^{rl}(\sigma^{T})_{r}^{k}\pa_{k}\chi_{u_n}^{\eps,\delta} \pa_{j}(\phi_\alpha h^{mi}) (\sigma^{T})_{i}^{j} (1-\chi_{v_n}^{\eps,\delta})d\mu d\xi \\
+ \int_{\mathbb{R}^{d+1}}(1-\chi_{v_n}^{\eps,\delta})\chi_{u_n}^{\eps,\delta}G(\mx,\xi)d\mu d\xi \\+
\int_{\mathbb{R}^{d+1}} \delta_{ml} \left(\phi_\alpha h^{rl}\pa_{k}((\sigma^{T})_{r}^{k}\chi_{u_n})\right)^{\eps,\delta} \left(h^{mi} \phi_\alpha \pa_j((\sigma^{T})_{i}^{j}(1-\chi_{v_n}))\right)^{\eps,\delta}d\mu d\xi.
\label{eq:num3.5}
\end{multline}
Putting together \eqref{eq:num3} and \eqref{eq:num3.5} and doing an integration by parts (to get the terms containing $\pa_{j}(1-\chi_{v_n}^{\eps,\delta})$ and $\pa_{k}\chi_{u_n}^{\eps,\delta}$, respectively, to cancel up to a term absorbed into the function $G$) gives
\begin{multline}
\int_{\mathbb{R}^{d+1}} \delta_{ml} \pa_{k}\left(h^{rl}(\sigma^{T})_{r}^{k}\tilde{\chi}_{u_n}\right)\star\rho_{\eps,\delta}\pa_{j}\left(h^{mi}(\sigma^{T})_{i}^{j}\bar{\chi}_{v_n}\right)\star\rho_{\eps,\delta}d\mu d\xi \\ \approx 
\int_{\mathbb{R}^{d+1}} \delta_{ml} \left(\phi_\alpha h^{rl}\pa_{k}((\sigma^{T})_{r}^{k}\chi_{u_n})\right)^{\eps,\delta} \left(h^{mi} \phi_\alpha \pa_j((\sigma^{T})_{i}^{j}(1-\chi_{v_n}))\right)^{\eps,\delta}d\mu d\xi \\
+\int_{\mathbb{R}^{d+1}}(1-\chi_{v_n}^{\eps,\delta})\chi_{u_n}^{\eps,\delta}G(\mx,\xi)d\mu d\xi.
\label{eq:num3.7}
\end{multline}
An analogous calculation for the second term from \eqref{eq:num2} establishes \eqref{eq:num4}.

\end{document}